\documentclass[letterpaper, 10 pt]{IEEEtran} 
\usepackage{cite}
\usepackage[tbtags]{amsmath}
\usepackage{amssymb,amsfonts}
\usepackage{amsthm}
\usepackage{algorithmic}
\usepackage{graphicx}
\usepackage{xcolor}
\usepackage[mathscr]{euscript}
\usepackage{mathrsfs}
\usepackage{tabularx}
\usepackage{multicol}
\usepackage{caption}
\usepackage[ruled,linesnumbered]{algorithm2e}
\usepackage[colorlinks = true, linkcolor = blue, citecolor = blue]{hyperref}
\usepackage{tcolorbox}
\usepackage{stackengine}
\usepackage{tikz}
\usepackage{cases}
\usepackage{lipsum}
\usepackage{enumitem}
\usepackage{color}
\usepackage{booktabs}

\DeclareMathOperator*{\argmax}{argmax}

\newcommand{\nonl}{\renewcommand{\nl}{\let\nl\oldnl}}

\definecolor{fuchsia}{HTML}{841184}
\definecolor{green}{HTML}{22CC22}
\definecolor{blue}{HTML}{0011EE}
\definecolor{red}{HTML}{FF111F}
\definecolor{gray}{HTML}{555555}

\newcommand{\ldef}{:=}

\newcommand{\Mc}[1]{\mathcal{#1}}
\newcommand{\real}{\ensuremath{\mathbb{R}}}

\newcommand{\natz}{{\mathbb{N}}_0}
\newcommand{\lbar}[1]{\stackunder[1.2pt]{$#1$}{\rule{.8ex}{.075ex}}}
\newcommand{\card}[1]{\lvert #1 \rvert}

\newcommand{\lanes}{\Mc{L}}
\newcommand{\tarr}{t^\mathcal{A}}
\newcommand{\tent}{t^\mathcal{E}}
\newcommand{\texit}{t^\mathcal{X}}
\newcommand{\tcoord}{t^\mathcal{C}}

\newcommand{\vset}{V}
\newcommand{\thor}{T_h}

\newcommand{\pretime}{T_{p}}
\newcommand{\coordhor}{T_c}
\newcommand{\coop}{c}
\newcommand{\vcoord}{\vset_{\coop}}
\newcommand{\vscheddone}{\vset_{s}}

\newcommand{\coordinterval}{\Mc{T}_{\coop}}

\newcommand{\follow}{q}
\newcommand{\vpbnd}{\Mc{V}}

\newcommand{\istar}{i^*}
\newcommand{\vfirst}{\Mc{F}}

\newcommand{\vfollow}{Q_i}
\newcommand{\wtime}{\tau_i}
\newcommand{\demand}{\Mc{D}}

\newcommand{\provphase}{\ensuremath{\texttt{prov\_phase}(i)}}
\newcommand{\coordphase}{\ensuremath{\texttt{coord\_phase}(\vcoord(k))}}

\makeatletter
\def\th@boldremark{\th@remark\thm@headfont{\bfseries}}
\makeatletter

\theoremstyle{boldremark}
\newtheorem{rem}{Remark}

\newcommand{\txt}[1]{\texttt{#1}}
\newcommand{\algcomm}[1]{\hfill \{\texttt{#1}\}}

\theoremstyle{plain}
\newtheorem{theorem}{Theorem}

\newcommand{\remend}{\relax\ifmmode\else\unskip\hfill\fi\hbox{$\bullet$}}

\graphicspath{{../figs/}{../figs/carlen4.5/}}
\begin{document}
\title{\LARGE{Data-Driven Distributed Intersection Management for Connected and Automated Vehicles}} \author{ \parbox{3 in}{\centering Darshan
    Gadginmath}
  \parbox{3 in}{ \centering Pavankumar Tallapragada }%
  \thanks{This work was partially supported by the Wipro IISc Research
    and Innovation Network.}  \thanks{Darshan Gadginmath is with
    the Department of Mechanical Engineering, University of California, Riverside
    {\tt\small\{dgadg001@ucr.edu\}}

     Pavankumar Tallapragada is with the Department of Electrical Engineering,
    Indian Institute of Science, Bangalore, India {\tt\small
      \{pavant@iisc.ac.in \}}} }

\maketitle

\begin{abstract}
  In this paper, we seek a scalable method for safe and efficient
  coordination of a continual stream of connected and automated
  vehicles at an intersection without signal lights. To handle a
  continual stream of vehicles, we propose trajectory computation in
  two phases - in the first phase, vehicles are constrained to not
  enter the intersection; and in the second phase multiple vehicles'
  trajectories are planned for coordinated use of the intersection.
  For computational scalability, we propose a data-driven method to
  obtain the intersection usage sequence through an online
  ``classification'' and obtain the vehicles' trajectories
  sequentially. We show that the proposed algorithm is provably safe
  and can be implemented in a distributed manner.  We compare the
  proposed algorithm against traditional methods of intersection
  management and against some existing literature through
  simulations. We also demonstrate through simulations that for the
  proposed algorithm, the computation time per vehicle remains
  constant over a wide range of traffic arrival rates.
\end{abstract}


\section{Introduction}

The advent of \emph{connected and autonomated vehicles} (CAVs)
presents an opportunity to rethink the problem of intersection
management. Further, in recent years, relevance of intersection
management has grown in non-traditional domains such as robot traffic
control in warehouses. The onboard sensing, computation and
communication capabilities that are available on CAVs or robots allow
us to do real time coordination of the CAVs or robots to achieve a
more efficient and un-signalized intersection management. In this
work, we propose a computationally efficient distributed algorithm and
a framework for offline data-driven tuning for the management of an
isolated intersection in the context of CAVs or robots.

\subsubsection*{Literature review}

Un-signalized intersection management has been
studied extensively in recent years using a variety of tools. The
survey papers \cite{LC-EC_2015,ZZ-NM-EEL-2020} highlight the different methods and tools
employed for un-signalized or autonomous intersection management.
Some early works~\cite{KD-PS_2008, DF-etal_2011, MH-TCA-PS_2011,
  DC-SDB-PS_2013} focused on reservation and multi-agent simulation
based algorithms. Some disadvantages of such solutions is that they
are computationally demanding, centralized and do not easily provide
insights into the system. Since then a major trend in the literature
has been to design model based, provably safe algorithms. For example,
\cite{HK-DC-PRK:2011} uses reservations for scheduling intersection
usage times, and \cite{AC-DDV-2015, HA-DDV-2018} (see also the
references therein) propose a supervisory control method where a
supervisor takes over only when a collision is imminent.
Another major trend in the field of autonomous intersection management
is the use of an optimal control framework for determining the
schedules and trajectories of the vehicles. \cite{YB-HAR-2019}
formulates an optimal control problem to optimize the trajectory of
each vehicle individually. However, intersection management problem in
general requires coordination of multiple vehicles through a combined
optimization. Such a combined optimization problem is
combinatorial. Thus, the overall problem of trajectory optimization
becomes a mixed integer
program~\cite{MWL-DR_2017,FA-ALF_2016,SAF-AV-2018,AM-etal-2019}. In
particular, the complexity of such formulations scales exponentially
with the number of vehicles. This limits the practical utility of the
exhaustive mixed integer program formulations as intersection
management is a time and safety critical application.

Given the computational complexity of the problem, along with the
motivation of designing distributed algorithms, several works have
sought to decompose the overall autonomous intersection management
problem into simpler sub-problems. \cite{AIMM-etal-2019_access,
  AIMM-etal-2018} together propose a high level intersection access
management by treating the vehicles on different lanes as queues and
use the idea of platooning for local vehicular control. Given
intersection usage schedule for the vehicles, \cite{AAM-CGC-YJZ-2018,
  YZ-CGC-2019} (and the references therein) seek to solve the
trajectory optimization problem in a decentralized manner by relaxing
the rear-end collision avoidance constraints and guarantee existence
of initial conditions under which the safety constraints are
satisfied. Further, these works also propose a method to drive the
vehicles to good ``initial conditions'' under which safety can be
guaranteed subsequently. \cite{MK-etal_2020,GRC-etal_2017_sequential,
  CL-etal_2017, DM-SK-2020, RH-etal-2018, PT-JC-2019, BL-etal-2019}
also decompose the problem into scheduling and trajectory
optimization. 
\cite{PT-JC-2019} proposes an algorithm, in which a central
intersection manager groups vehicles into bubbles and schedules the
bubbles as a whole to use the intersection. Given the schedule, the
vehicles compute provably safe trajectories using a distributed
switched controller. 
\cite{YW-HC-FZ-2019} proposes to achieve coordination of vehicles by
optimizing a notion of joint rewards for the vehicles. For this, it
employs Q-learning where the joint actions are found using the
$\epsilon$-greedy approach. The learnt actions are stored in Q-tables
which can be used by the vehicles when the system is deployed. Since
the Q-tables are learnt from episodic data with a focus on minimizing
the intersection delay, the resulting trajectories can turn out to be
non-smooth. Although the vehicles can obtain near-optimal joint
actions, the paper does not provide safety guarantees.
  
Comfort of passengers and generation of smooth trajectories for
vehicles is another area of interest in autonomous intersection
management. \cite{MB-etal_2015} surveys driver comfort in autonomous
vehicles and highlights the inadequate research on passenger comfort
in path and motion planning of autonomous vehicles.
\cite{IAN-etal-2016_comfort} studies the problem of vehicles merging
into highways and uses a model predictive control architecture that
optimizes comfort by minimizing the squares of both acceleration and
jerk. \cite{PD-etal-2016} introduces a metric of comfort which is a
combination of vehicle-jitter, jerk and deviation from a desired
velocity.  \cite{YZ-CGC-2019} and \cite{RH-etal_2019_turns} focus on
vehicles turning at the intersection and impose a curvature-based
acceleration constraint to capture the comfort of the passengers.

\subsubsection*{Contributions}

In this paper, we propose a computationally scalable algorithm for
coordinating and optimizing the trajectories of a continual stream of
CAVs at and near an isolated, un-signalized autonomous
intersection. Most of the papers in the literature consider only the
problem of coordinating a fixed set of vehicles. Applying such
solutions to a continual stream of vehicles can result in inefficiency
or feasibility/safety itself may be violated in the long run. Further,
the optimal coordination of vehicles at an intersection is a mixed
integer problem, which scales badly with the number of vehicles and
lanes. This work is the only paper, to the best of our knowledge, that
addresses these issues systematically.

The key insight behind the proposed data-driven framework is that the
computation of the optimal intersection usage sequence can be thought
of as an online classification problem from a space of features that
encode the ``demand'' of a vehicle and the traffic following it to the
vehicle's precedence for using the intersection. With this insight, we
design a computationally very efficient and scalable data-driven
algorithm for the problem of autonomous intersection management. The
proposed framework also has the ability to incorporate both ``micro''
information about the individual vehicles' state as well as ``macro''
information such as traffic arrival rates. Such a combination again
has not been explored in the literature. This element of our framework
is particularly useful under high traffic arrival rates.

The second set of major contributions of this paper relates to how we
handle a continual stream of vehicles. Most papers in the area
essentially propose a one-shot algorithm with the implicit suggestion
that the one-shot algorithm should be run repeatedly. However, not
explicitly considering the continual stream of vehicles could in
general lead to loss of feasibility of safe trajectories.
In the proposed framework of this paper, we split the trajectory of
each vehicle into two phases (1) \emph{provisional phase} and (2)
\emph{coordinated phase}. Every vehicle operates in the provisional
phase as soon as it enters the system and it is restricted from
entering the intersection. Periodically, the vehicles in the
provisional phase obtain a trajectory for their coordinated phase and
start executing them. This framework thus offers a complete for a
continual stream of vehicles while ensuring safety, feasibility and
near optimality of the solutions.

Lastly, we evaluate the performance of our algorithm through an
extensive collection of simulations. In particular, we compare our
algorithm with that of an ``optimal'' algorithm, signalized
intersection management, first-in first-out based intersection
management as well as the algorithm proposed in~\cite{PT-JC-2019}. We
also demonstrate the computational efficiency of the proposed
algorithm through some coarse metrics. In particular, we observe from
the simulations that for the proposed algorithm the computation time
per vehicle essentially remains constant for a wide range of traffic
arrival rates. This is in contrast to an exhaustive optimization
algorithm for which the computation time per vehicle scales
exponentially with the traffic arrival rate.

\subsubsection*{Notation}

We use $\real$ and $\natz$ for the set of real and whole numbers,
respectively. For a discrete set $V$, we let $|V|$ be the cardinality
of the set $V$.

\section{Model and Problem Formulation}\label{sec:model}

\subsection{Model}

\subsubsection*{Geometry of the Region of Interest} 

We consider an isolated intersection with a set of lanes $\lanes$.
%
%
We refer to the lanes together with the intersection itself as the
\emph{region of interest}. Every lane $l \in \lanes$ has a unique
fixed path $P_l \subset \real^2$ associated with it. We denote by
$s_{l}$ the length of the portion of the path $P_l$, that lies within
the intersection. The length of all the paths leading to the
intersection is $d$.  Along the path on lane $l$, we let the positions
at the beginning of the region of interest, the beginning of the
intersection and the end of the intersection be $-d$, $0$ and $s_l$,
respectively. Figure~\ref{fig:intersection} presents the basic
geometry of an example region of interest, where the set of lanes is
$\lanes = \{1,2,3,\dots,12\}$ with 3 lanes (for going left, straight
and right) on each branch. Figure~\ref{fig:intersection} labels lanes
1, 8 and 12 and skips the rest for clarity. The translucent box shaded
in red represents the conflict region or the intersection, which is
the region of potential inter-lane collisions.
\begin{figure}[htb]
  \centering \includegraphics[width = 0.9\linewidth]{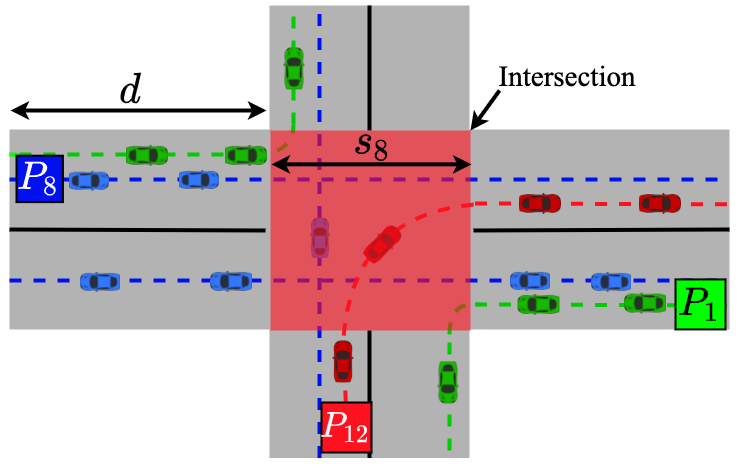}
  \caption{Region of interest and the geometry of the
    intersection. Here only 3 lanes with numbers 1, 8 and 12 have been
    labeled.  }
  \label{fig:intersection}
\end{figure}

As it is apparent from Figure~\ref{fig:intersection}, 
the paths of some lanes do not
intersect while others do. For each pair of lanes $l, m \in \lanes$,
we define a notion of \emph{compatibility} $c(l,m)$ as
\[
  c(l,m) \ldef \begin{cases}
    1 , & P_{l} \cap P_{m} = \emptyset \\
    0 , & P_{l} \cap P_{m} \neq \emptyset, \ P_l \neq P_m \\
    -1, & P_{l} = P_{m} .
  \end{cases}
\]
We say that a pair of lanes $l$ and $m$ are \emph{compatible} if
$c(l,m) = 1$ and \emph{incompatible} if $c(l,m) = 0$. In the sequel,
we require that vehicles on incompatible lanes not be in the
intersection at the same time. Note that our framework can handle
other configurations of incoming branches and intersections by
appropriately defining the compatibility pairs $c(l,m)$. We choose the
standard configuration, shown in Figure~\ref{fig:intersection}, purely
for ease of exposition. We assume that vehicles do not change lanes
within the region of interest.



We assume that all vehicles are CAVs - they can communicate with each
other and the infrastructure, and are automated. We also assume that
the vehicles do not change lanes within the region of interest.  The
CAVs can enter the region of interest on any lane in $\lanes$. We
denote the lane that vehicle $i$ traverses on by $l_i \in \lanes$ and
the vehicle's length by $L_i$. The state of the vehicle at time $t$ is
$(x_i(t), v_i(t))$, where $x_i$ and $v_i$ are the position of the
front bumper of the vehicle and the vehicle's velocity respectively on
the path $P_{l_i}$. The dynamics of the vehicle $i$ is
\begin{equation}
  \begin{aligned}
    &\dot{x}_i(t) = v_i(t) , \quad \dot{v}_i(t) =
    u_i(t), \label{eqn:DI}
  \end{aligned}
\end{equation}
where $u_i$ is the acceleration input to vehicle $i$. The vehicles are
in the region of interest for different time durations. In particular,
vehicle $i$ enters the region of interest at the \emph{arrival time},
$\tarr_i$, enters the intersection at the \emph{entry time},
$\tent_i$, and leaves the intersection at the \emph{exit time},
$\texit_i$. Thus, $x_i(\tarr_i) = -d$, $x_i(\tent_i) = 0$. and
$x_i(\texit_i) = s_{l_i}$.

\subsection{Problem}

The aim of the autonomous intersection management problem
is to compute safe trajectories for the CAVs while maximizing the
following objective function
\begin{equation}
  J \ldef \sum_{i \in \vset} \!\!
  \int \displaylimits_{\tarr_i}^{\tarr_i+\thor} \!\! \left[
    W_v v_i(t)  - \big(W_a  u_i^2(t) + W_j  \dot{u}_i^2(t)\big)
  \right] \mathrm{d}t , \label{eqn:cf1}
\end{equation}
where $\vset$ is the set of all vehicles that arrive in the region of
interest during a time interval of interest, $\dot{u}_i$ is the
\emph{jerk} of vehicle $i$ and $W_v$, $W_a$ and $W_j$ are non-negative
weights. We model the instantaneous discomfort of the passengers in
vehicle $i$ by the linear combination of the squares of acceleration
and jerk. This metric penalizes sporadic high-magnitude disturbances
caused by braking and acceleration manoeuvres performed by a
vehicle~\cite{MB-etal_2015}. Thus, each vehicle's contribution to the
objective function is a linear combination of the distance it travels
and the comfort of passengers in a time horizon $\thor$, starting from
the vehicle's arrival time $\tarr_i$.

\begin{rem}[Objective function]
  In the objective function~\eqref{eqn:cf1}, the contribution of
  vehicle $i$ is a linear combination of the distance traversed
  (integral of the velocity) and the comfort (negative of discomfort)
  experienced by the passengers of the vehicles. For comfort, we
  consider the acceleration and jerk only in the longitudinal
  direction along a vehicle's path. For vehicles turning left or
  right, we ignore the fact that forces also act in the lateral
  direction. However, the principles we illustrate in this paper could
  easily be extended to also consider `lateral' comfort. We seek to
  maximize the social (over all vehicles) objective function of
  traversal distance and comfort. Maximizing traversal distance over a
  fixed horizon in each term of~\eqref{eqn:cf1} is a proxy for
  minimizing traversal time for crossing the intersection. We choose
  this indirect metric because directly minimizing the traversal time
  results in a problem with a variable and unknown horizon for each
  vehicle. On the other hand, the fixed horizon
  formulation~\eqref{eqn:cf1} provides a computational advantage in
  the online construction of constraints for a stream of vehicles.
  \remend
\end{rem}

\subsubsection*{Constraints}

The first set of constraints on the CAVs are bounds on their
acceleration $u_i(t)$ and velocity $v_i(t)$, i.e.,
\begin{equation}
  u_i(t) \in [\lbar{u}, \bar{u}], \quad v_i(t) \in [\lbar{v},
  \bar{v}], \label{eqn:u-v-constr}
\end{equation}
for all $\forall i \in V$ over an appropriate time interval. 
We assume that $\lbar{u} < 0$, and $\lbar{v} = 0$.

The second set of constraints ensure safety between the vehicles.  Two
types of collisions can occur in the region of interest: (1)
rear-ended collision between successive vehicles on the same lane, and
(2) collision between vehicles on incompatible lanes, within the
intersection. To ensure in-lane safety we impose a safe-following
distance between any two successive vehicles travelling on the same
lane. Consider two vehicles $i$ and $j$ on the same lane ($l_i = l_j$)
such that $i$ is the vehicle immediately following $j$.  This
arrangement is formally denoted using the \emph{follower indicator
  function} as,
\begin{equation*}
  \follow(i,j) \ldef
  \begin{cases}
    1, \! & \text{if } l_i = l_j, \ x_i < x_j, \\
    &\ \nexists k \text{
      s.t. } l_k = l_i, \ x_i < x_k < x_j  \\
    0, \ & \text{otherwise} .
  \end{cases}
\end{equation*}

The minimum safe-following distance $D$ between vehicles $i$ and $j$,
when $q(i,j)=1$, is a function of the velocities of the two vehicles
and is given by~\cite{PT-JC:2018-tcns, PT-JC-2019},
\begin{equation}
  D(v_i,v_j) = L_j + r + \max \left\{ 0, \frac{ 1 }{ -2\lbar{u} } \left(
      v_i^2(t)-v_j^2(t) \right) \right\} . \label{eqn:rear-end distance}
\end{equation}
Here, $r$ is a robustness parameter which is a constant distance to account for measurement and
communication errors and delays. Then the \emph{rear-end safety
  constraint} is
\begin{align}
  & x_j(t) - x_i(t) \geq D(v_i,v_j), \ j \text{ s.t. }
   \follow(i,j) = 1  \label{eqn:rearend_safety}
\end{align}
for the time interval of interest. Note that the rear-end safety
constraint~\eqref{eqn:rearend_safety} is more robust to loss of
coordination, either due to breakdown in communication, control or due
to malicious vehicles, than rear-end non-collision
constraints~\cite{PT-JC:2018-tcns, PT-JC-2019}. To ensure safety
within the intersection, we additionally impose the constraint that
vehicles on incompatible lanes cannot be within the intersection
simultaneously. Thus, the \emph{intersection safety constraint} for a
pair of vehicles $i$ and $k$ is
\begin{align}
  \tent_i \geq \texit_k \ \ \mathrm{ \textbf{OR} } \ \
  \tent_k \geq \texit_i, \ \text{ if } c(l_i,l_k)
  = 0.  \label{eqn:int-safety}
\end{align}

Then, the proposed optimal control problem for intersection management
is as follows,
\begin{subequations}
  \begin{align}
    & \underset{u_i(.), \ i \in \vset}{\max} \  J  \\
    \mathrm{s.t.} \ %
    & \eqref{eqn:DI} , \eqref{eqn:u-v-constr}, \eqref{eqn:rearend_safety}\
      \forall t \in [\tarr_i,\tarr_i+\thor], \forall i \in \vset \\
    \ &\eqref{eqn:int-safety} \ \forall i,k \in \vset \
        \mathrm{s.t.} \ c(l_i,l_k) = 0 .
  \end{align}\label{eqn:main_Prob}
\end{subequations}


\begin{rem}[Challenges in solving \eqref{eqn:main_Prob} and problem
  statement]
  There are several challenges in solving
  Problem~\eqref{eqn:main_Prob}. First, vehicles arrive randomly in a
  stream into the system and the information about their arrival and
  state is revealed only incrementally. Thus,
  Problem~\eqref{eqn:main_Prob} cannot be ``solved'' in the usual
  sense. Hence, we seek an algorithm that satisfies the
  constraints in the problem and we utilize~\eqref{eqn:cf1} as a
  metric for evaluating the performance of an algorithm after it makes
  all the decisions. Further, although the exact arrival times of the
  vehicles are not known a priori, we allow for the knowledge of the
  statistical data such as the mean arrival rate of vehicles. We seek
  to leverage this information for more efficient traffic
  management. Second, Problem~\eqref{eqn:main_Prob} is a mix of large
  scale optimal control and combinatorial optimization. In particular,
  the number of optimal control sub-problems that
  constraints~\eqref{eqn:int-safety} generate scales exponentially
  with the number of vehicles and lanes. This is a serious issue since
  intersection management is a time and safety critical
  problem. Hence, we seek algorithms that are computationally scalable
  and yet provide near optimal performance.  \label{rem:complexity}
  \remend
\end{rem}

\section{Overview of the Algorithm} \label{sec:PropAlgo}

Considering the complexities and time-criticality associated with
Problem \eqref{eqn:main_Prob}, we propose a computationally efficient
algorithm to compute a sequence for intersection usage as well as the
trajectories for the vehicles. To overcome the randomness in the
arrival of traffic, and the challenges associated with incremental
revelation of information, we split the trajectory of each vehicle
into two phases: \emph{provisional} and \emph{coordinated}. The
provisional phase begins as soon as a vehicle arrives into the
region of interest. The vehicle seeks to maximize its objective under
the constraint of a safe approach towards the intersection. At a
prescribed time, the vehicle switches to its coordinated phase from
its provisional phase. The vehicles in their coordinated phase use the
intersection safely while aiming to optimize the overall objective.

In this section, we give an overview of the proposed algorithm to
solve Problem~\eqref{eqn:main_Prob}. For ease of exposition, we
initially assume the presence of a central \emph{intersection manager}
(IM) that has communication and computation capabilities with which it
carries out the coordination of the traffic. At the end of
Section~\ref{sec:coop_plan_methods}, we discuss how essentially all
the functions of the IM can be carried out in a distributed manner. We
present the overview of the algorithm in two parts: from the
perspectives of an arbitrary vehicle $i$ and the IM in
Algorithm~\ref{algo:vehicle_persp}, and in
Algorithm~\ref{algo:IM_persp}, respectively.

\subsection{Vehicle $i$'s Perspective}

A vehicle $i$ starts execution of Algorithm~\ref{algo:vehicle_persp}
at $\tarr_i$, its time of arrival into the region of interest.
\begin{algorithm}
\caption{Algorithm from a vehicle $i$'s perspective}
\label{algo:vehicle_persp}
\SetAlgoLined

\If {$t = \tarr_i$}{ %
  \txt{receive} $\tcoord_i$ \txt{and trajectory of vehicle preceding $i$
    in its lane}\\
  \provphase \hfill \\
  \txt{send provisional trajectory to IM}\\
  \txt{start provisional phase} }
    
\If {$t = \tcoord_i$}
  { 
    \txt{receive new trajectory from IM for coordinated phase}\\
    \txt{start coordinated phase}
  }
\end{algorithm}
Vehicle $i$ communicates with the IM as soon as it arrives at
$\tarr_i$. The IM prescribes $\tcoord_i$, the \emph{start time of
  coordination phase} for vehicle $i$, and also informs about the
planned trajectory of the vehicle (if any) that precedes vehicle $i$
on its lane. This is sufficient for vehicle $i$ to plan its trajectory
for the provisional phase, which ends at $\tcoord_i$. In particular,
vehicle $i$ computes its trajectory for the provisional phase by
solving optimal control Problem~\eqref{eqn:prob_provplan}, which we
refer to in Algorithm~\ref{algo:vehicle_persp} as \provphase. Vehicle
$i$ communicates its provisional phase trajectory back to the IM and
starts executing it at $\tarr_i$. At $\tcoord_{i}$, vehicle $i$
receives a new trajectory from the IM for the coordinated phase.

\subsubsection*{Provisional Phase of Vehicle $i$}

Here we describe \provphase, the method that vehicle $i$ utilizes to
compute the trajectory for its provisional phase. At $\tarr_i$,
vehicle $i$ obtains $\tcoord_i$, the start time of its coordination
phase, and the trajectory of the vehicle preceding it on its
lane. Vehicle $i$ computes an optimal trajectory under several
constraints including the \emph{intersection entry prevention
  constraint},
\begin{align}
  v_i(t) \leq \vpbnd(x_i(t)) \ldef \sqrt{2 \lbar{u}
  x_i(t)} , \label{eqn:provplanvelbound}
\end{align}
for all $t$ in the time interval of interest. The upper bound
$\vpbnd(x_i(t))$ is the maximum velocity that vehicle $i$ may have at
position $x_i(t)$ so that with maximum braking ($u_i(t) = \lbar{u}$)
vehicle $i$ can come to a stop before entering the intersection. Thus
this constraint prevents the vehicle from entering the intersection
under the bounded control constraint. Then, the optimal control
problem for vehicle $i$'s provisional phase is
\begin{align}
  &\underset{u_i(.)}{\max} \ 
  \int_{t^\mathcal{A}_i}^{t^\mathcal{A}_i+\pretime} \Big(W_v  v_i(t)
  - \big[W_a \ u_i^2(t) + W_j \ \dot{u}_i^2(t)\big]\Big) \ \mathrm{d}t
  \notag\\
  &\mathrm{ s.t. } \ %
  \eqref{eqn:DI} , \eqref{eqn:u-v-constr}, \eqref{eqn:rearend_safety},
  \eqref{eqn:provplanvelbound}  \
  \forall t \in [\tarr_i,\tarr_i+\pretime]. \label{eqn:prob_provplan}
\end{align}

\subsection{Intersection Manager's Perspective}

Now, we describe Algorithm~\ref{algo:IM_persp}, which is from the IM's
perspective.
\begin{algorithm}
\caption{Algorithm from IM's perspective}
\label{algo:IM_persp}
\SetAlgoLined

\If {$t = \tarr_i$}{%
  $\tcoord_i \gets k \coordinterval$, with
  $k = \min \{ k \in \natz : k
  \coordinterval \geq \tarr_i \}$ \\
  \txt{send to vehicle $i$,} $\tcoord_i$ \txt{and trajectory of
    vehicle
    preceding $i$ in its lane} \\
  \txt{receive vehicle $i$'s provisional trajectory }
}
    
\If {$t = k\coordinterval$} {%
  $\vcoord(k) \gets \{ i \ : \ \tarr_i \in \big( \ (k-1) \coordinterval, \ k \coordinterval \ \big]  \}$\\
  \coordphase \\
  \txt{send trajectories to vehicles}  $\vcoord(k)$ \\
  $k \gets k + 1$ \\
}

\end{algorithm}
As soon as a vehicle $i$ enters the region of interest, the IM sends
to vehicle $i$, the next instance of coordinated trajectory planning
as $\tcoord_i$ and the trajectory of the vehicle preceding vehicle $i$
on its lane $l_i$ so that vehicle $i$ can compute its provisional
phase trajectory and communicate it back to the IM. In this paper, for
simplicity, we assume that IM carries out coordinated planning
periodically at the instances $k\coordinterval$, where $k \in
\natz$. And we let $\tcoord_i = k \coordinterval$, where $k$ is the
smallest integer such that $k \coordinterval \geq \tarr_i$.

At each coordinated trajectory planning time instance
$k \coordinterval$, the IM first considers $\vcoord(k)$, the set all
the vehicles that have arrived during the interval
$\big( \ (k-1) \coordinterval, \ k \coordinterval \ \big]$. Then, it
computes a trajectory for each vehicle in $\vcoord(k)$ seeking to
achieve optimized coordination and ensures the vehicles cross the
intersection safely. We denote the coordinated phase planning problem
at the instance $k\coordinterval$ by \coordphase. The IM communicates
the trajectories for the coordinated phase to the vehicles in
$\vcoord(k)$, which then execute them.

In Section~\ref{sec:coop_plan_methods}, we present \coordphase, the
algorithm for planning the trajectories in the coordinated phase.

\begin{rem}[Computation instances]
  The coordination planning instances need not be periodic and may be
  adapted to the traffic.
  Also, in Algorithms~\ref{algo:vehicle_persp}
  and~\ref{algo:IM_persp}, we have presented various functions to be
  executed at specific time instances. But, this is purely for ease of
  exposition and one could easily modify these algorithms to account
  for computational delays. \remend
\end{rem}

In Section~\ref{sec:coop_plan_methods}, we present \coordphase, the
algorithm for planning the trajectories in the coordinated phase.

\section{Coordinated Phase}
\label{sec:coop_plan_methods}

This section presents the trajectory optimization for the coordinated
phase. We first present \emph{combined optimization}, which is a naive
centralized method and is not computationally scalable. Building on
this method, we present the data-driven sequential weighted algorithm,
that is significantly superior in terms of computational requirements.
Moreover, as we demonstrate through simulations in
Section~\ref{sec:sims}, this algorithm performs almost as well as the
combined optimization.

The planning for the coordinated phase is carried out periodically
with period $\coordinterval$. In particular, at the instance
$k \coordinterval$, trajectory planning is carried out for the set of
vehicles $\vcoord(k)$ that arrive into the region of interest during
the interval
$\big( \ (k-1) \coordinterval, \ k \coordinterval \ \big]$. In this
section, we discuss the methods for coordinated planning at an
arbitrary but fixed instance $k \coordinterval$. We also introduce the
set $\vscheddone$ that contains all the \emph{vehicles that have
  received a trajectory for the coordinated phase}. The vehicles in
$\vcoord(k)$ are added to $\vscheddone$ after they receive their
respective trajectories for the coordinated phase. For brevity, we
omit the argument $k$ for $\vcoord(k)$ in the rest of this
section. Further, notice that $\tcoord_i$ is the same for all vehicles
in $\vcoord(k)$. Hence, in the sequel, we drop the index $i$ from
$\tcoord_i$.

\subsection{Combined Optimization}

In this method, the IM computes the trajectories for all the vehicles
in $\vcoord$ simultaneously. The optimal control problem for the
combined optimization method is a variation of the problem \eqref{eqn:main_Prob}. 
The only differences are: the time horizon for
the coordinated phase is $\coordhor$ and the set of participating
vehicles is $\vcoord$. Specifically, the objective function is
\begin{equation*}
  J^{\coop} = \sum_{i \in \vcoord} 
  \int\displaylimits_{\tcoord}^{\tcoord + \coordhor}
  \Big( W_v \ v_i(t)  - \big[W_a \ u_i^2(t) + W_j \
  \dot{u}_i^2(t)\big] \Big)\ \mathrm{d}t 
\end{equation*}
and the optimal control problem for combined optimization is
\begin{subequations}
  \label{eqn:prob_combinedoptimization}
  \begin{align}
    & \underset{u_i(.), \ i \in \vcoord}{\max} \  J^{\coop}  \\
    \text{s.t.} \ %
    & \eqref{eqn:DI} , \eqref{eqn:u-v-constr}, \eqref{eqn:rearend_safety} \
      \forall t \in [\tcoord,\tcoord + \coordhor], \ \forall i \in
      \vcoord \\
    \ &\eqref{eqn:int-safety} \ \forall i,k \in \vcoord \cup \vscheddone \
        \mathrm{s.t.} \ c(l_i,l_k) = 0
        . \label{const:combined_opti_intersection_safety}
  \end{align}
\end{subequations}
Subsequent to solving~\eqref{eqn:prob_combinedoptimization} and
updating the trajectories for the vehicles, $\vscheddone$ is updated
to $\vscheddone \cup \vcoord$.

Combined optimization~\eqref{eqn:prob_combinedoptimization} requires
the IM to compute optimal trajectories for each feasible intersection
usage sequence and then pick the best sequence and the corresponding
optimal trajectories. However, the number of feasible sequences grows
exponentially with the number of vehicles on incompatible lanes. Thus,
this method is not scalable and is not well suited for the time and
safety critical problem of autonomous intersection management. Hence,
we next propose a computationally scalable and efficient method for
computing near optimal sequences and trajectories for the coordinated
phase.


\subsection{Data Driven Sequential Weighted Algorithm (DD-SWA)}
\label{subsect:DDSWA}

We now propose a scalable method for optimizing the intersection usage
sequence and trajectories of vehicles in the coordinated phase. We
call it \emph{data-driven sequential weighted algorithm}
(DD-SWA). 
The method scales linearly with the number of vehicles and as we
discuss in the sequel, this method is very amenable to a distributed
implementation. We present an overview of DD-SWA in
Algorithm~\ref{algo:DD-SWAalgo}. The algorithm begins with the set of
unscheduled vehicles,~$\vcoord$. In Step~\ref{stp:vfirst}, we identify
$\vfirst$, the set of vehicles in $\vcoord$ that are closest to the
intersection.
\begin{algorithm}
\caption{DD-SWA}
\label{algo:DD-SWAalgo}
\let\oldnl\nl
\If {$t = 0$}{
$\vscheddone \gets \emptyset$ \algcomm{set of scheduled vehicles} \\
} 
\If {$t = k\coordinterval$ , for $k \in \natz$ ,}{
  \While{$|\vcoord| > 0$}{
          $\vfirst \gets \{ i \in \vcoord \ | \ x_i \geq x_j, \
          \forall j \in \vcoord
          \text{ s.t. } l_j = l_i\}$ \label{stp:vfirst}\\
          \For{$i \in \vfirst$}{ %
            $p_i \gets $ \txt{precedence}($i$) \label{stp:precedence}
          }%
          $\istar \gets \argmax \{ p_i \ \vert \ i \in \vfirst \}$ \\
          \txt{traj\_opti}$(\istar)$
    \label{stp:traj-opti} 
      \\
      $\vcoord \gets \vcoord \setminus \istar $ \algcomm{remove
        $\istar$ from $\vcoord$ }\\
      \nonl{$\vscheddone \gets \vscheddone \cup \istar$
        \algcomm{$\istar$ is scheduled} \\} } }
\end{algorithm}
In Step~\ref{stp:precedence} of the algorithm, the \emph{precedence
  index} $p_i$ is computed for every vehicle $i$ in $\vfirst$.
The vehicle $\istar \in \vfirst$ with the highest precedence index
(after arbitrarily resolving any potential ties) is allowed to use the
intersection before any other vehicle in $\vfirst$. A trajectory for
the coordinated phase is then computed for $\istar$ in
Step~\ref{stp:traj-opti}, after which $\istar$ is removed from
$\vcoord$ and added to $\vscheddone$. This process is repeated
until $\vcoord$ is empty. The vehicles optimize their
trajectories sequentially so as to satisfy the intersection safety
constraint~\eqref{eqn:int-safety}. Next, we describe the computation
of the precedence indices \txt{precedence}($i$) and the trajectory
optimization.

\subsubsection{Computation of the Precedence Index
  \txt{precedence}(\texorpdfstring{$i$}{})}

We let the precedence index $p_i$ be a linear combination of certain
\emph{scheduling features} related to the vehicle $i \in \vfirst$
\begin{align}
  p_i & \ldef w_x (d+x_i(\tcoord_i)) + w_v v_i(\tcoord_i) + w_t
        (\tcoord_i - \tarr_i) + w_n \card{ \vfollow } + \notag\\
      & w_s \frac{\sum_{j \in \vfollow} (x_i(\tcoord_i) -
        x_j(\tcoord_i))}{\card{\vfollow}} + w_{\sigma} \ \sigma_{l_i}
        - w_w \ \tau_i . \label{eqn:precedenceindex}
\end{align}
Three of the scheduling features are based on the state at time
$\tcoord_i$ and history of the vehicle $i$, namely, distance traveled
since arrival $d+x_i(\tcoord_i)$, velocity $v(\tcoord_i)$ and time
since arrival $(\tcoord_i - \tarr_i)$ of vehicle $i$. Three features
capture the ``demand'' on lane $l_i$ that is ``following'' vehicle
$i$. First of these features is the number of vehicles
$\card{\vfollow}$, where $\vfollow$ is the \emph{set of vehicles that
  follow vehicle $i$ on lane $l_i$ at time $\tcoord_i$}. The second
feature is the average separation of vehicles in $\vfollow$ from
vehicle $i$. The third feature in this group is the average rate of
arrival of vehicles $\sigma_{l_i}$ on lane $l_i$. The final feature is
the \emph{minimum wait time to use the intersection}, $\wtime$, for
vehicle $i$. Specifically,
\begin{equation*}
  \wtime \ldef \max \{ \texit_m - \tcoord_i \ | \ m \in
  \vscheddone \ s.t. \ c(l_i,l_m) = 0 \} .
\end{equation*}
All the weights in \eqref{eqn:precedenceindex} are positive, which
means minimum wait times effect the precedence index negatively. This
is to prevent frequent switching of the right of way between
incompatible lanes. The weighted linear combination of the features
makes the computation of the precedence indices extremely simple. In
this work, we propose tuning the weights based on offline simulations.

\subsubsection{Trajectory Optimization}

In Algorithm~\ref{algo:DD-SWAalgo}, the trajectories of the vehicles
are computed sequentially. In each iteration, the vehicle $\istar$
with the greatest precedence index is selected for trajectory
optimization. However, notice from the combined optimization
problem~\eqref{eqn:prob_combinedoptimization} that the optimization of
the trajectory of vehicle $i$ is coupled to the optimization of the
other vehicles' trajectories through the constraints. One of the
purposes of the precedence indices is to set a precedence in the
constraint~\eqref{eqn:int-safety}. Even then, the coupling is not
fully eliminated. In order to compute the trajectories sequentially,
we seek to decouple~\eqref{eqn:prob_combinedoptimization} into several
optimization problems - one per vehicle. We have to do this in a
manner that ensures we still get near optimal solutions
to~\eqref{eqn:prob_combinedoptimization}. Such a method aids in
developing a distributed algorithm.

A natural starting point for constructing such a decoupled problem is
to consider only the term involving $\istar$ in $J$
of~\eqref{eqn:prob_combinedoptimization}. However, this ignores the
``demand'' for the intersection usage. Thus, we seek to modify the
``marginal'' cost function of the vehicle $\istar$ by incorporating a
measure of the demand. We first introduce a notion of \emph{demand},
$\demand_i$, from vehicle $i$ and those following it on the lane
$l_i$. Specifically,
\begin{equation*}
  \demand_i \ldef p_i + w_w \tau_i .
\end{equation*}
Then, we let the objective function for generating a trajectory for
vehicle $\istar$ to be
\begin{align}
  J_{{\istar}}^{\coop} = %
  & \int\displaylimits_{\tcoord_{\istar}}^{\tcoord_{\istar} + \coordhor} \Big(
    \overline{W}_{v} v_{\istar}(t)  - \left[W_a u_{\istar}^2(t) + W_j
    \dot{u}_{\istar}^2(t) \right] \Big) \mathrm{d}t
    , \label{eqn:SWA_objfun}
\end{align}
where
$\displaystyle \overline{W}_{v} \ldef w_l \frac{\sum_{i \in \vfirst}
  \demand_{i}}{|\vfirst|} W_v$ and $w_l$ is a scaling factor. Then,
\txt{traj\_opti}($\istar$) in Step~\ref{stp:traj-opti} of Algorithm
\ref{algo:DD-SWAalgo} is
\begin{subequations}
  \begin{align}
    \underset{u_{i^*}(.)}{\max} \ J_{\istar}^{\coop} \text{ s.t. } \ %
    &\eqref{eqn:DI} , \eqref{eqn:u-v-constr},
      \eqref{eqn:rearend_safety} \ \forall t \in
      [\tcoord_{\istar},\tcoord_{\istar}+\coordhor] \\
    &\tent_{\istar} \geq \tau_{\istar} +
      \tcoord_{\istar} , \label{const:DDSWA_intersection_safety}
  \end{align} \label{eqn:traj-opti-DDSWA}
\end{subequations}
with $i = \istar$ in~\eqref{eqn:DI} , \eqref{eqn:u-v-constr} and
\eqref{eqn:rearend_safety}.

\begin{rem}[Computational complexity of DD-SWA]
  In DD-SWA, we obtain the intersection usage order by computing the
  precedence indices for vehicles in $\vfirst$ as a weighted linear
  combination of the scheduling features and selecting the maximizer
  of the precedence indices. Also, note that
  $\card{\vfirst} \leq \card{\lanes}$, the number of lanes. These
  aspects make the computation of the intersection usage order very
  simple. Further, the computation of the trajectory of the vehicles
  is also of significantly lesser complexity since for each vehicle we
  essentially need to solve an optimal control problem in which the
  only decision variables are those related to the vehicle itself and
  the constraints are significantly simplified. In the sequel, we use
  simulations to demonstrate that DD-SWA performs only marginally
  worse compared to combined optimization while the computation time
  per vehicle essentially stays constant for a wide range of traffic
  arrival rates. On the other hand for combined optimization, the
  computation time per vehicle increases exponentially with the
  traffic arrival rate. \remend
\end{rem}

We see that in all the three optimal control
problems~\eqref{eqn:prob_provplan},
\eqref{eqn:prob_combinedoptimization} and~\eqref{eqn:traj-opti-DDSWA}
feasibility implies safety. In the following result we show that if
the vehicles arrive into the region of interest in a safe
configuration then they are always in a safe configuration (both
in-lane and within the intersection) for all time during the
provisional as well as the coordinated phases.

\begin{theorem}[Sufficient condition for system wide inter-vehicle
  safety] \label{thm:feasiblity}%
  If every vehicle $i$ satisfies the rear-end safety constraint
  \eqref{eqn:rearend_safety} at the time of its arrival, $\tarr_i$,
  and its initial velocity is such that that
  $v_i(\tarr_i) \leq \min\{\bar{v},\vpbnd(-d)\}$, feasiblility of
  problems \eqref{eqn:prob_provplan},
  \eqref{eqn:prob_combinedoptimization} and
  \eqref{eqn:traj-opti-DDSWA} is guaranteed. Consequently, safety of
  all the vehicles is also guaranteed for all time.
\end{theorem}
\begin{proof}
  The assumption that each vehicle $i$ satisfies the rear-end safety
  constraint~\eqref{eqn:rearend_safety} at $\tarr_i$ ensures that
  there exists a control trajectory to ensure rear-end safety with the
  vehicle that precedes $i$ on its lane $l_i$. Further, the assumption
  that $v_i(\tarr_i) \leq \min\{\bar{v},\vpbnd(-d)\}$ implies that
  there exists a control trajectory that ensures that the vehicle can
  come to a stop before the beginning of the intersection. Thus, the
  optimization problem for the provisional phase
  \eqref{eqn:prob_provplan} is feasible.

  If problem \eqref{eqn:prob_provplan} is feasible, the trajectory for
  the provisional phase guarantees that vehicle $i$ satisfies the
  rear-end safety constraint~\eqref{eqn:rearend_safety} and the
  intersection entry prevention
  constraint~\eqref{eqn:provplanvelbound} at $\tcoord_i$, the start
  time of the coordinated phase of vehicle $i$. This property ensures
  that the intersection safety constraints
  \eqref{const:combined_opti_intersection_safety} and
  \eqref{const:DDSWA_intersection_safety} for combined optimization
  and DD-SWA respectively are also feasible as the vehicle can come to
  a stop before the beginning of the intersection if necessary. Thus,
  feasibility of problem \eqref{eqn:prob_provplan} guarantees the
  feasibility for problems \eqref{eqn:prob_combinedoptimization} and
  \eqref{eqn:traj-opti-DDSWA}. Since feasiblility of the problems
  ensure rear-end safety and intersection safety, safety between every
  pair of vehicles is also guaranteed.
\end{proof}

\subsection{Distributed Implementation of the Algorithm}

Note that each vehicle $i$ can implement its provisional phase in a
distributed manner by communicating only with the vehicle preceding it
in its lane. The design of DD-SWA is also amenable to a distributed
implementation. The information required to calculate a vehicle's
precedence index and to solve its trajectory optimization problem can
be obtained with distributed communication.


Each vehicle can obtain information such as distance travelled,
velocity and time since arrival locally. 
The other scheduling features, the safety constraints
\eqref{eqn:rearend_safety} and \eqref{eqn:int-safety} and the weights
for the scheduling features require communication. We make a
distinction between the types of communication required for this
purpose. The three types of communication required are: (1)
intra-lane, (2) inter-lane and (3) central communication. In
intra-lane communication, a vehicle $i \in \vfirst$ needs to
communicate with only a vehicle $j$ such that $\follow(i,j) = 1 $ or
$\follow(j,i) = 1 $, \textit{i.e.}, $i$ needs to communicate with just
the vehicles immediately preceding or following it on its lane $l_i$.
The number of vehicles following $i$, $\card{\vfollow}$, can be
counted in a distributed manner and can be communicated from one car
to the next in $\vfollow$, the vehicles following $i$ and ultimately
to the vehicle $i$ itself.  
Similarly, the vehicle immediately in front of $\istar$ can
communicate its position and velocity trajectory which are sufficient
to compute \eqref{eqn:rearend_safety}. The intersection safety
constraint and the minimum wait time feature require $\istar$ to
communicate and receive the exit time of the vehicle on an
incompatible lane. We denote such communication as inter-lane
communication. Lastly,
intersection-specific information such as the weights for the
scheduling features $w_x,\dots,w_s$ and the average arrival rate of
traffic $\sigma_{l_i}$ need to be communicated to the vehicles in
$\vfirst$ from a central infrastructure, such as an IM. Thus the
central infrastructure's or IM's function is essentially restricted to
communication.

\section{Simulations} \label{sec:sims}

To evaluate the proposed algorithm, a simulation framework using
Casadi \cite{casadi_2018} and Python was created. All the simulations
were performed on an Intel i9-9900k 3.6GHz processor with 128GB of
RAM. In the simulations, we assume that vehicles arrive according to a
Poisson process with an average arrival rate of $\sigma_l$ on lane
$l \in \lanes$.  To evaluate the proposed algorithm, we compare
combined optimization and DD-SWA against a signalized intersection,
the Hierarchical-Distributed algorithm \cite{PT-JC-2019}
and the coordinated phase with a first-in first-out (FIFO) protocol
for the sequence of intersection usage. The simulation results are for
the particular case where vehicles only pass straight across the
intersection. However, the proposed algorithms hold
even when turning is allowed. We present the algorithms and the
comparisons in greater detail below\footnote{A video describing the main features of the proposed algorithm and
simulations is available at:
\url{http://www.ee.iisc.ac.in/people/faculty/pavant/files/figs/Data-Driven-IM.mp4}.}.

\subsubsection*{Arrival of Vehicles}

Recall that the existence of feasible trajectories for the provisional
and coordinated phases is guaranteed if the conditions mentioned in
Theorem~\ref{thm:feasiblity} are satisfied. Thus to ensure
feasibility, we restrict the maximum velocity of the vehicles at the
time of their arrival, to $\min\{\bar{v},\vpbnd(-d)\}$. In the
simulations here, we assume that vehicles arrive according to a
Poisson process with an average arrival rate of $\sigma_l$ on lane
$l \in \lanes$. However, a specific realization of arrival times of
the vehicles may cause a violation of the rear-end safety constraint
at the arrival time itself. To avoid this, we check the separation
between the successive vehicles at the time of arrival. If the
constraint \eqref{eqn:rearend_safety} is violated, the arrival of the
vehicle is delayed until the constraint is satisfied.

To evaluate the proposed algorithm, we compare combined optimization
and DD-SWA against a signalized intersection, the
Hierarchical-Distributed algorithm proposed in \cite{PT-JC-2019} and
the coordinated phase with a first-in first-out (FIFO) protocol for
the sequence of intersection usage. The simulation results that we
present here are for the particular case where vehicles only pass
straight across the intersection, \textit{i.e.} vehicles arrive in
lanes $ l \in \lanes_s = \{2,5,8,11\}$. However, the proposed
algorithms hold even when turning is allowed. We present the
algorithms and the comparisons in greater detail below.

\subsubsection*{Signalized Intersection}
In this algorithm, every vehicle $i$ that enters the region of
interest performs \txt{prov\_phase}($i$) to approach the intersection.
When a lane $l$ receives a green signal, all the vehicles in lane $l$
are considered to be a part of $\vcoord$ and they are given a
\emph{green trajectory} to exit the intersection by solving problem
\eqref{eqn:prob_combinedoptimization}. The cycle times and green times
for the signals are obtained using Webster's method
\cite{TU-etal_2015} corresponding to the arrival rate $\sigma_l$ in
each lane.



\subsubsection*{Hierarchical-Distributed (HD) Algorithm}
This is the algorithm presented in \cite{PT-JC-2019}. 

\subsubsection*{First-In First-Out (FIFO)}
We use a FIFO protocol for the coordinated phase in these simulations.

\subsubsection*{Simulation Parameters}
Table \ref{tab:genparam} lists the parameters of the intersection and
the vehicles that are common to all the algorithms.
\begin{table}[!htb]
\renewcommand{\arraystretch}{0.9}
\small
\caption{ General Simulation Parameters}
\label{tab:genparam}
\begin{tabularx}{0.49\textwidth}{l@{\extracolsep{\fill}} c c}
  \hline
  \multicolumn{3}{ c }{Intersection Parameters}\\
  \hline 
  \textbf{Parameter} & \textbf{Symbol} & \textbf{Value} \\
  \hline
  Length of branch & $d$ & $60 \ m$\\
  Length of intersection (Straight) & $s_l$ & $20 \ m$ \\ 
  Length of vehicle & $L_i$ & $4.3 \ m$ \\ 
  Robustness parameter \eqref{eqn:rear-end distance} & $r$ & $0.2$ m  \\
  Min. Acceleration & $\lbar{u}$ & $-3 \ m/s^2$ \\
  Max. Acceleration & $\bar{u}$ & $3 \ m/s^2$ \\  
  Max. Velocity & $\bar{v}$ & $11.11 \ m/s$ \\
  \hline
  \multicolumn{3}{ c }{Proposed Algorithm Parameters}\\
  \hline
  Time interval for coordinated phase & $\coordinterval$ & $3 \ s$ \\
  Time horizon for provisional phase & $\pretime$ & $\tcoord_i - \tarr_i$\\
  Time horizon for coordinated phase & $\coordhor $ & $30 \ s$ \\
  Time horizon for objective function \eqref{eqn:cf1} & $\thor$ & $30 \ s$   \\
  \hline
\end{tabularx}
\end{table}
We conducted simulations using
all the algorithms for several arrival rates of traffic. We chose the
simulation time for each simulation to be equal to the time duration
of 10 cycles of a signalized intersection corresponding to the
particular arrival rate $\sigma$ obtained from the Webster's method~\cite{TU-etal_2015}. In each of the
simulations, we conducted 20 trials for each of the algorithms for
each arrival rate $\sigma$. Then, we compared the average time to
cross and average objective function value per vehicle over the 20
trials for each value of $\sigma$ across all the algorithms.

\subsection{Results}

We present 3 sets of comparisons between the various algorithms
mentioned previously. Table \ref{tab:weights} indicates the weights on
the scheduling features used for computing the precedence
index~\eqref{eqn:precedenceindex} in DD-SWA in each of the comparative
simulations.

\begin{table*}[!htb]
\centering
\caption{Weights for comparisons}
\renewcommand{\arraystretch}{1}
\begin{tabular}{l c c c c }

\hline
\multicolumn{5}{ c }{\textsc{General Weights}}\\
\hline
 \textbf{Parameter} & \textbf{Symbol} & \textbf{Comparisons 1 and 2} & \textbf{Comparison 3} & \textbf{Comparison 4}  \\ 
 \hline
	Weight on acceleration term  & $W_a$  & $0$ & $0$ & $1$ \\
	Weight on jerk term          & $W_j$  & $0$ & $0$ & $1$ \\
	Weight on velocity term      & $W_v$  & $1$ & $1$ & $1$ \\
\hline
\multicolumn{5}{ c }{ \textsc{DD-SWA Weights}}\\
\hline
  \textbf{Scheduling Features} & \textbf{Weight Symbol} & \textbf{Comparison 1 and 2} & \textbf{Comparison 3} & \textbf{Comparison 4}\\
\hline
	Distance travelled since arrival                  & $w_x$        & $0.1$  & $0.5$  & $0.8$  \\ 
	Velocity                                          & $w_v$        & $5$    & $4$    & $7$    \\
	No. of vehicles following $i$ in $l_i$            & $w_n$        & $4.5$  & $6$    & $5$  \\
	Time since arrival                                & $w_t$        & $3$    & $3$    & $5$  \\
	Average arrival rate in $l_i$                     & $w_{\sigma}$ & $40$   & $65$   & $40$ \\
	Average separation of vehicles from $i$ in $l_i$  & $w_s$        & $6$    & $7$    & $7$  \\
	Minimum wait time to use the intersection         & $w_w$        & $0.5$  & $1$    & $5$  \\
	Average demand scaling factor                     & $w_l$        & $0.02$ & $0.02$ & $0.02$ \\ 
\hline
\end{tabular}
	
\label{tab:weights}
\end{table*}
It also indicates the weights on the scheduling features used in
DD-SWA in each of the comparative simulations. Next, we discuss in
detail each of the four comparisons.

\subsubsection{Comparison 1}

In this comparison, there is no weight on comfort of the passengers in
the objective functions and the vehicles aim to only maximize the
distance travelled. In particular, in the ``running cost'', the
weights on acceleration and jerk terms ($W_a$ and $W_j$ respectively)
are set to 0 and the weight on velocity term ($W_v$) is set to 1 for
trajectory optimization problems for the provisional phase and the
coordinated phase. In the HD algorithm, the fuel cost represented by
$F_i(\bar{v}_i)$ in Equation (5) of \cite{PT-JC-2019} is set to 0 so
that the vehicles only aim to minimize the time spent within the
intersection. This ensures a fair comparison between the HD algorithm
and other algorithms. We compare the average time to cross (TTC) for
the vehicles under the different algorithms. We show simulation
results for arrival rates ($\sigma$) in the range of $0.01$ to $0.09$
vehicles/s per lane with an increment of $0.01$ vehicles/s per
lane. Figure~\ref{fig:comp123}(a) shows that the average TTC for
vehicles with combined optimization and DD-SWA is comparable for all
arrival rates in the considered range. FIFO performance is marginally poor
compared to DD-SWA and Combined Optimization as we are only considering
low arrival rates. The HD algorithm's performance
is comparable for arrival rates between $0.01$ and $0.03$ vehicles/s
per lane but performs poorly beyond $0.04$ vehicles/s. The signalized
algorithm performs better than the HD algorithm after $0.08$
vehicles/s. 
\begin{figure*}[!htb]

\vspace*{0.6cm}

\begin{tikzpicture}[overlay, remember picture]
\node[anchor=north west, 
      xshift=2.3cm, 
      yshift=0.8cm] 
     at (0,0) 
     {\includegraphics[width=13.5cm]{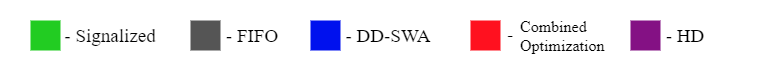}}; 
\end{tikzpicture}

\begin{multicols}{4}

\resizebox{4.8cm}{4.8cm}{
\begin{tikzpicture}
  \node (img1)  {\includegraphics[scale=0.15]{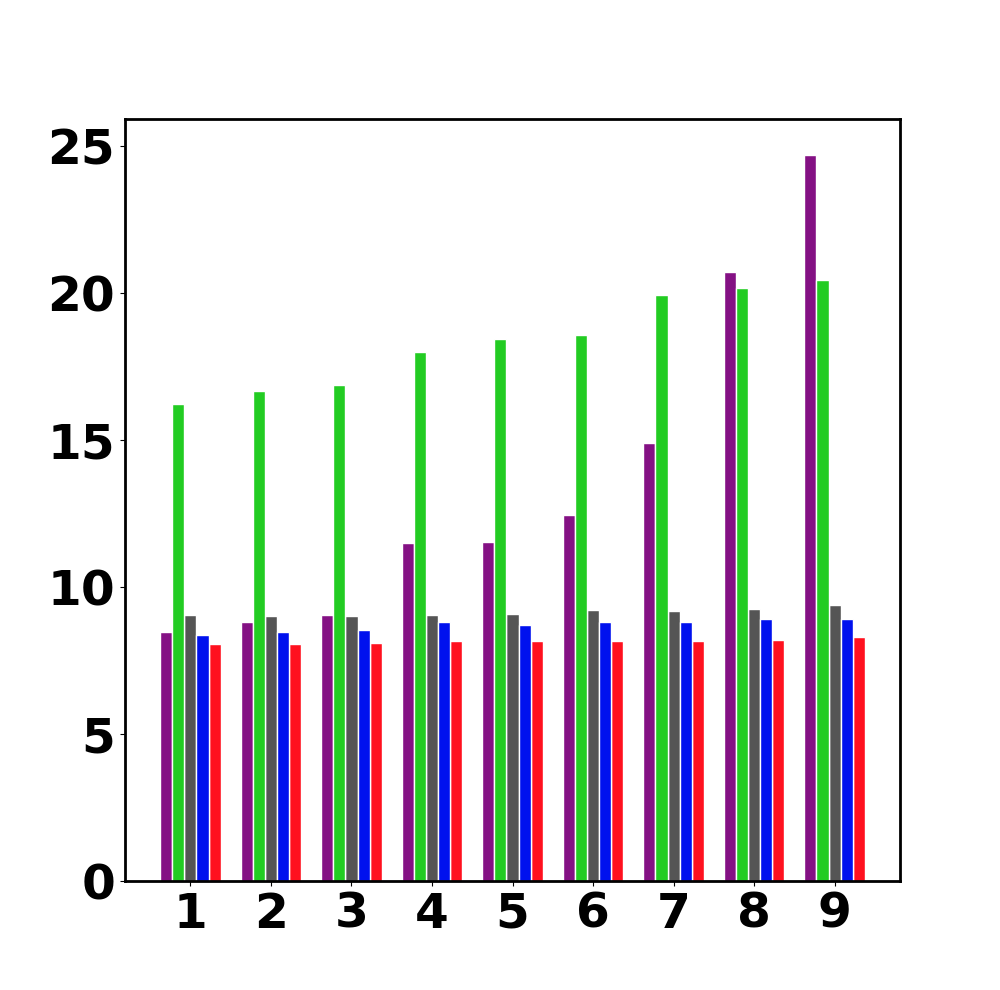}};
  \node[below of= img1, node distance=0cm, yshift=-2.0cm,font=\color{black}]  {\small $100\sigma$ (veh./s/lane)};
  \node[left of= img1, node distance=0cm, rotate=90, anchor=center,yshift=2.1cm,font=\color{black}] { \small Avg. TTC (s)};
\end{tikzpicture}
}
\caption*{(a)}
\resizebox{4.8cm}{4.8cm}{
\begin{tikzpicture}
  \node (img1)  {\includegraphics[scale=0.15]{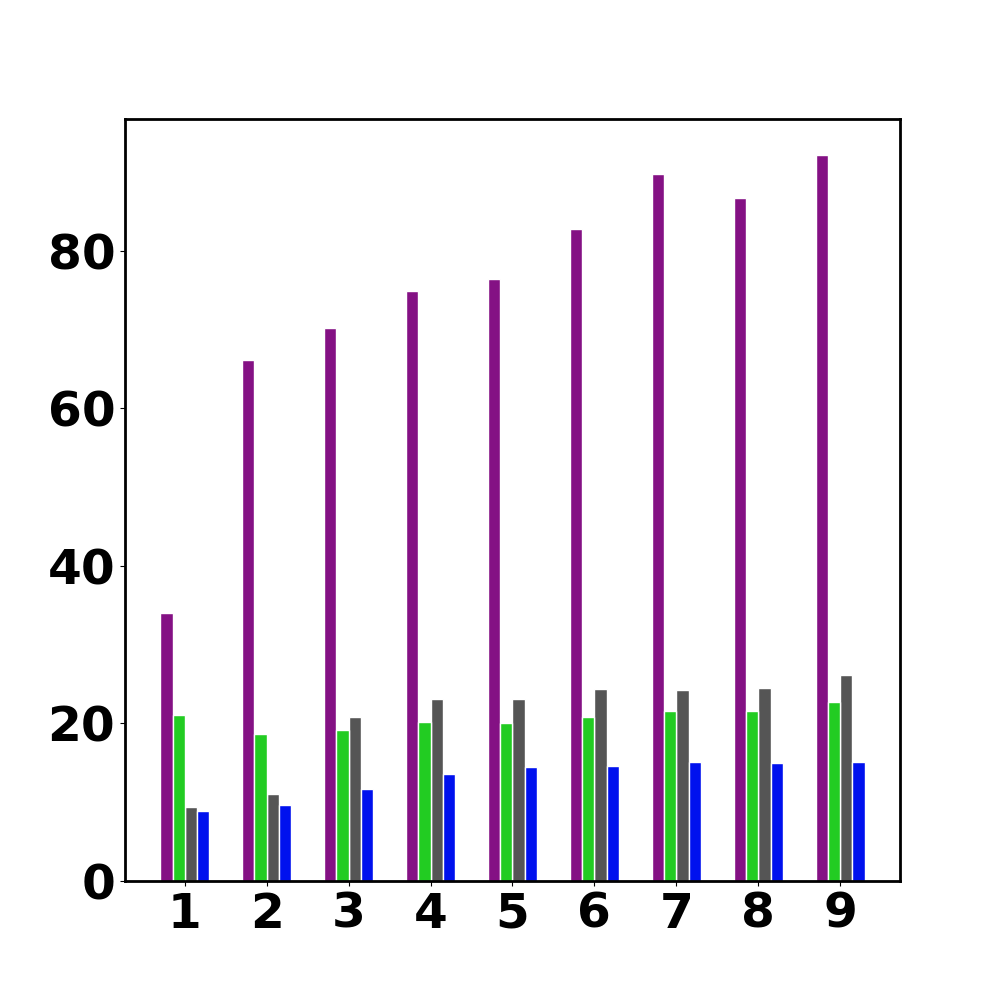}};
  \node[below of= img1, node distance=0cm, yshift=-2.0cm,font=\color{black}]  {\small  $10\sigma$ (veh./s/lane)};
  \node[left of= img1, node distance=0cm, rotate=90, anchor=center,yshift=2.1cm,font=\color{black}] {\small Avg. TTC (s)};
\end{tikzpicture}
}
\caption*{(b)}

\resizebox{4.8cm}{4.8cm}{
\begin{tikzpicture}
  \node (img1)  {\includegraphics[scale=0.15]{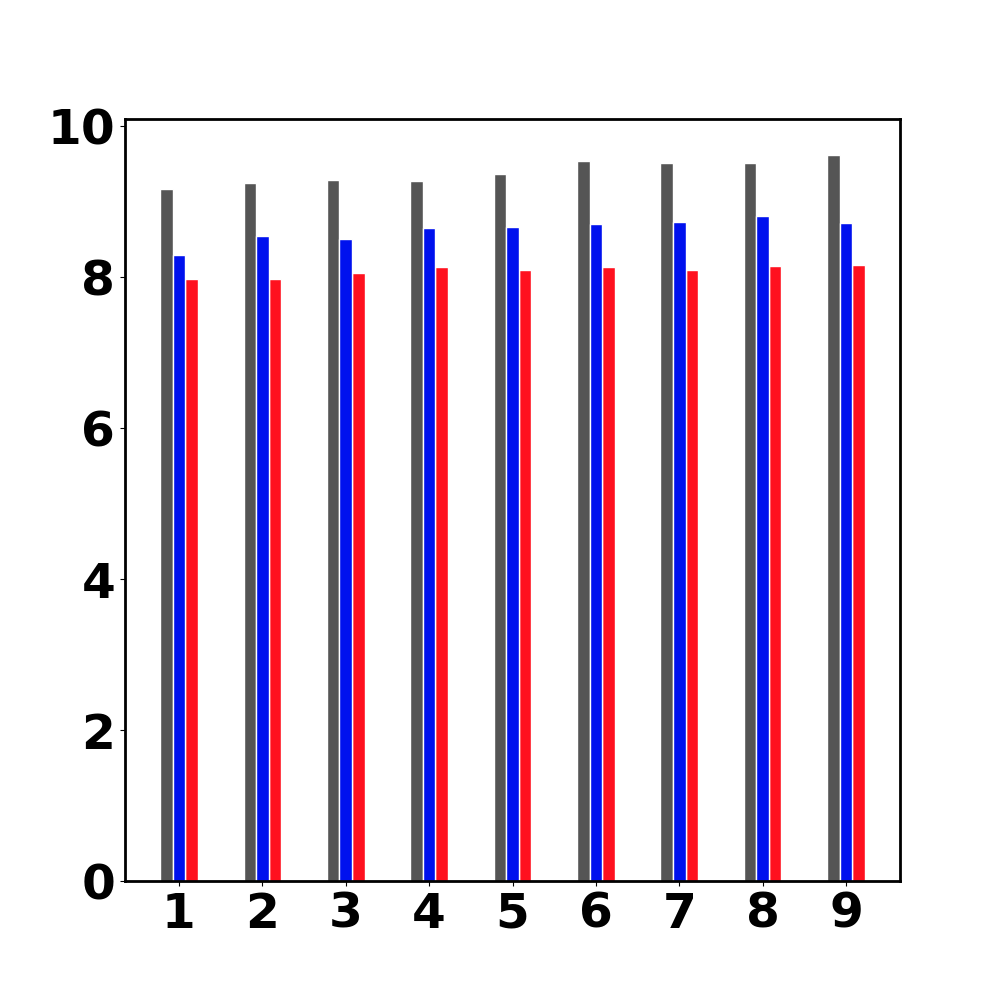}};
  \node[below of= img1, node distance=0cm, yshift=-2.0cm,font=\color{black}]  {\small $100\sigma$ (veh./s/lane)};
  \node[left of= img1, node distance=0cm, rotate=90, anchor=center,yshift=2.1cm,font=\color{black}] { \small Avg. TTC (s)};
\end{tikzpicture}
}
\caption*{(c)}
\resizebox{4.8cm}{4.8cm}{
\begin{tikzpicture}
  \node (img1)  {\includegraphics[scale=0.15]{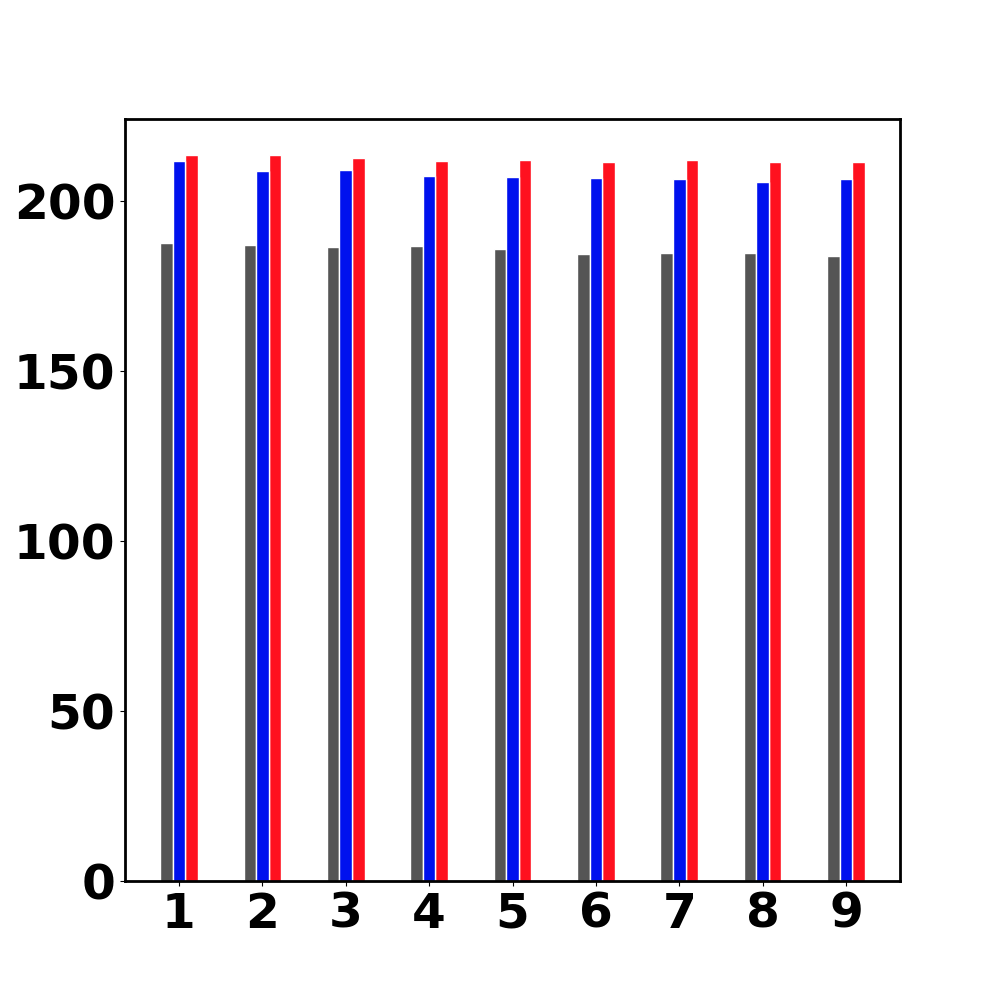}};
  \node[below of= img1, node distance=0cm, yshift=-2.0cm,font=\color{black}]  {\small $100\sigma$ (veh./s/lane)};
  \node[left of= img1, node distance=0cm, rotate=90, anchor=center,yshift=2.1cm,font=\color{black}] {\small Avg. objective value};
\end{tikzpicture}
}
\caption*{(d)}

\end{multicols}
\vspace*{-0.5cm}

\caption{Results of Comparisons 1, 2 and 3 for various arrival rates
  $\sigma$. In Comparisons 1 and 2, the arrival rate is homogeneous
  across all lanes. In Comparison 3, the arrival rate is
  inhomogeneous. (a) Comparison 1 - the average time to cross (TTC)
  for low arrival rates. (b) Comparison 2 - the average time to cross
  (TTC) for high arrival rates.
  (c) Comparison 3 - the average time to cross (TTC) for the
  vehicles. (d) Comparison 3 - the average objective value, which is
  the average distance traversed by the vehicles from the time of
  their arrival. }
\label{fig:comp123}
\end{figure*}

\subsubsection{Comparison 2}

In Comparison 2, the objective is again to maximize only the distance
travelled by the vehicles, but the comparison is made for arrival
rates ($\sigma$) from $0.1$ to $0.9$ vehicles/s per lane. As the
computation time for combined optimization is significantly higher
compared to the other algorithms, we choose to not include it for
comparison 2. Figure \ref{fig:comp123}(b) shows that DD-SWA continues
to perform better than all the other algorithms. Although the time to
cross initially increases for DD-SWA, it saturates at $0.4$ vehicles/s
per lane. FIFO is initially better than the signalized intersection 
until $0.2$ vehicles/s but it's performance rapidly deteriorates as 
the arrival rate increases. The signalized algorithm outperforms FIFO at $0.3$
 vehicles/s and outperforms the HD algorithm at all arrival rates in this range.
However, it does not perform better than the DD-SWA. The HD algorithm performs
significantly worse than the other algorithms due to it's nature of 
creating bubbles with multiple vehicles.

\subsubsection{Comparison 3}

In Comparison 3, we compare between combined optimization, DD-SWA and FIFO
when the arrival of traffic is inhomogeneous, \textit{i.e.}, when the
arrival rates are not the same for all the lanes in consideration. In
particular, we set $\sigma_2 = \sigma_8 = \sigma$ and
$\sigma_5 = \sigma_{11} = \frac{\sigma}{2}$ for different values of
$\sigma$. We again set the weights on acceleration and jerk terms to 0
and the weight on the velocity term to 1 in this comparison. Figure
\ref{fig:comp123}(c) shows the average time to cross for the vehicles
for combined optimization and DD-SWA. The performance of DD-SWA is
only marginally poor compared to that of combined optimization, but FIFO
performs significantly poorly compared to both the algorithms. Figure
\ref{fig:comp123}(d) shows the average objective value~\eqref{eqn:cf1}
over a period of $\tarr_i$ to $\tarr_i+\thor$ for every vehicle. To
compute the objective value, only the vehicles that crossed the
intersection by the end of the simulation time were considered to be
in the set $\vset$ in the objective function \eqref{eqn:cf1}.  The
average objective value per vehicle is depicted in
\ref{fig:comp123}(d).  Note that the weights on the scheduling
features in DD-SWA were tuned to improve its performance.

\subsubsection{Comparison 4}

Combined optimization, DD-SWA and FIFO are compared against each other
 with weights on the velocity, acceleration and jerk
terms are all set to 1. Figures \ref{fig:comp45}(a) and (b) depict the
average TTC and the average objective value for the two methods from
$\tarr_i$ to $\tarr_i + \thor$. A decrease in the average objective
value and an increase in the average TTC can be observed as there is
an emphasis on both comfort and the distance travelled by the
vehicles. It can be observed that combined optimization marginally
outperforms DD-SWA both in terms of the average objective value and
the average TTC. FIFO performs significantly poorly compared to the other two algorithms. Similar to the legend in Figure \ref{fig:comp123},
the blue and red bars in Figure \ref{fig:comp45} correspond to DD-SWA
and combined optimization respectively.

\begin{figure}[!htb]

\vspace*{0.6cm}

\begin{tikzpicture}[overlay, remember picture]
\node[anchor=north west, 
      xshift=0.0cm, 
      yshift=0.8cm] 
     at (0,0) 
     {\includegraphics[width=9cm]{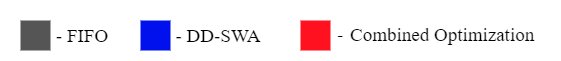}}; 
\end{tikzpicture} 

\begin{multicols}{4}

\resizebox{4.8cm}{4.8cm}{
\begin{tikzpicture}
  \node (img1)  {\includegraphics[scale=0.15]{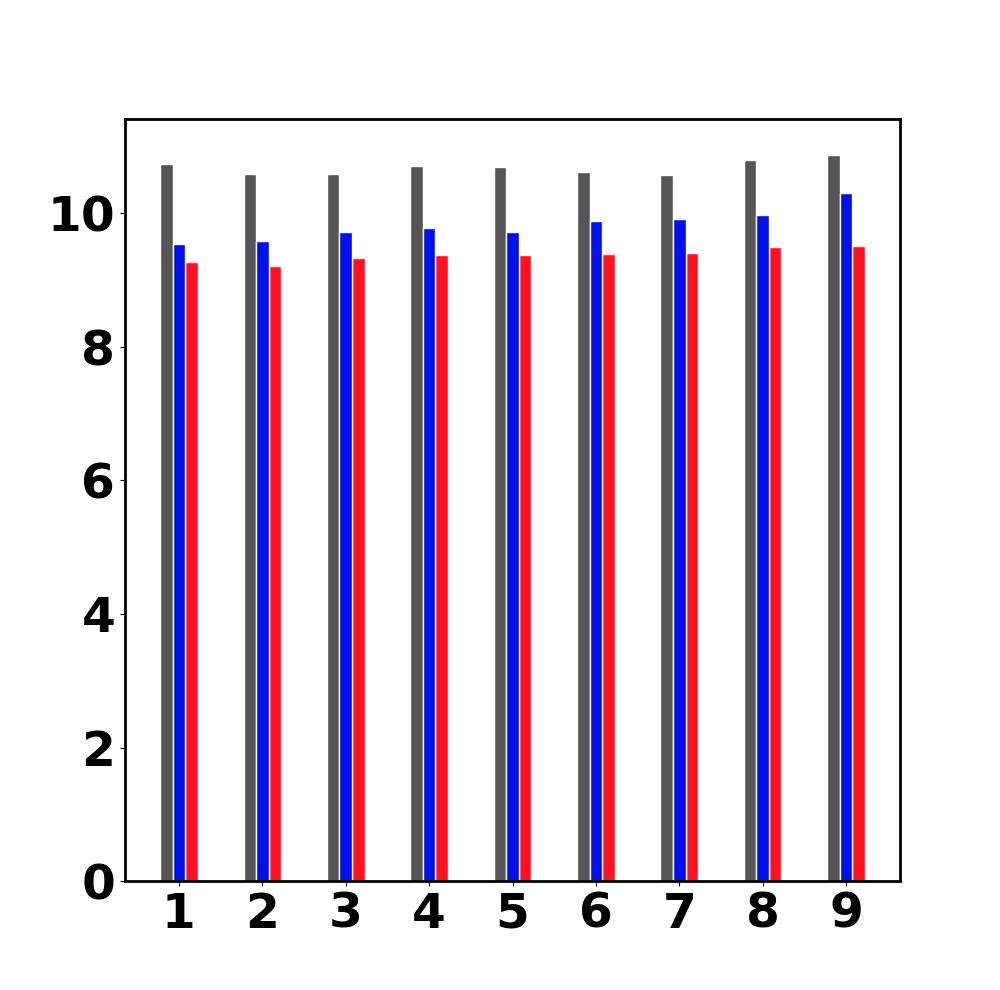}};
  \node[below of= img1, node distance=0cm, yshift=-2.0cm,font=\color{black}]  {\small  $100\sigma$ (veh./s/lane)};
  \node[left of= img1, node distance=0cm, rotate=90, anchor=center,yshift=2.0cm,font=\color{black}] {\small  Avg. TTC (s)};
\end{tikzpicture}
}
\caption*{(a)}
\resizebox{4.8cm}{4.8cm}{
\begin{tikzpicture}
  \node (img1)  {\includegraphics[scale=0.15]{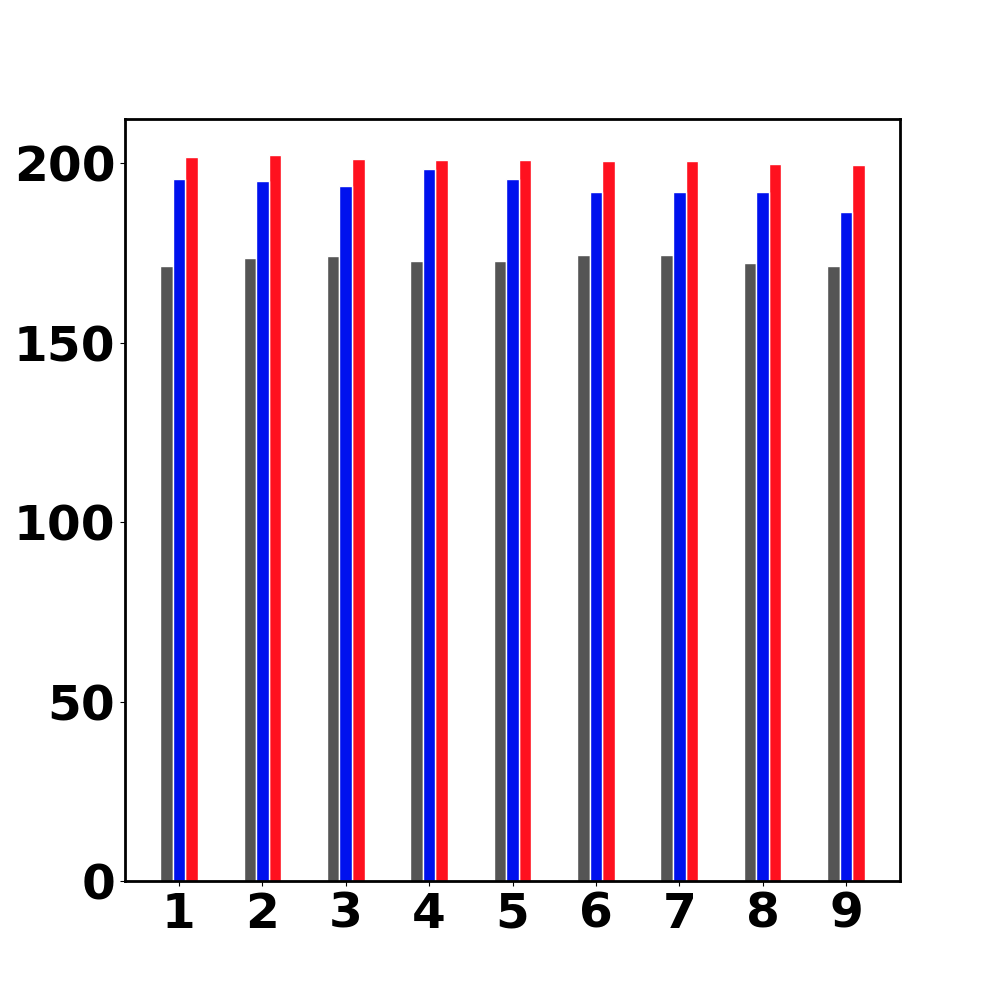}};
  \node[below of= img1, node distance=0cm, yshift=-2.0cm,font=\color{black}]  {\small  $100\sigma$ (veh./s/lane)};
  \node[left of= img1, node distance=0cm, rotate=90, anchor=center,yshift=2.1cm,font=\color{black}] {\small  Avg. objective value};
\end{tikzpicture}
}
\caption*{(b)}


\end{multicols}
\vspace*{-0.5cm}
\caption{Results of Comparison 4. Here, the weights on the velocity,
  acceleration and jerk terms are equal to 1.}
\label{fig:comp45}
\end{figure}



\subsubsection*{Computation time comparison}

We compare the computation time per vehicle for combined optimization
and DD-SWA to emphasize the computational advantage of DD-SWA. We
initially compare the size of $\vcoord$ for every round of trajectory
optimization for the coordinated phase. We inspect the variation in
$|\vcoord|$ using box plots. In the Figure~\ref{fig:comptime}, The
lower and the upper edges of the boxes represent the first quartile
($Q1$, $25^{th}$ percentile) and the third quartile \big($Q3$,
$75^{th}$ percentile) respectively. The whiskers above and below the
boxes represent the maximum and the minimum of the data. The maximum
is calculated as $Q3+1.5(Q3-Q1)$ and the minimum as
$Q1-1.5(Q3-Q1)$. Data beyond the maximum and the minimum are
considered as outliers and they are represented by small circles. The
mean of the data is represented by the bold black
line. Figures~\ref{fig:comptime}(a) and (b) show the box plots of
$|\vcoord|$ for combined optimization and DD-SWA respectively and
Figures~\ref{fig:comptime} (c) and (d) compare the computation time
per vehicle for combined optimization and DD-SWA for various arrival
rates. Although the trend of $|\vcoord|$ is similar for both the
algorithms, the trend of computation time per vehicle is significantly
different. The computation time for combined optimization increases
exponentially as $|\vcoord|$ increases with the arrival
rate. Figure~\ref{fig:comptime}(d) shows that DD-SWA has a nearly
constant value of computation time per vehicle. We attribute this to
the low computational effort required to determine the sequence of
intersection usage and for sequentially optimizing the trajectories of
the vehicles in $\vcoord$.

\begin{figure*}[!htb]
\begin{multicols}{4}
\centering

\resizebox{4.8cm}{4.8cm}{
\begin{tikzpicture}
  \node (img1)  {\includegraphics[scale=0.15]{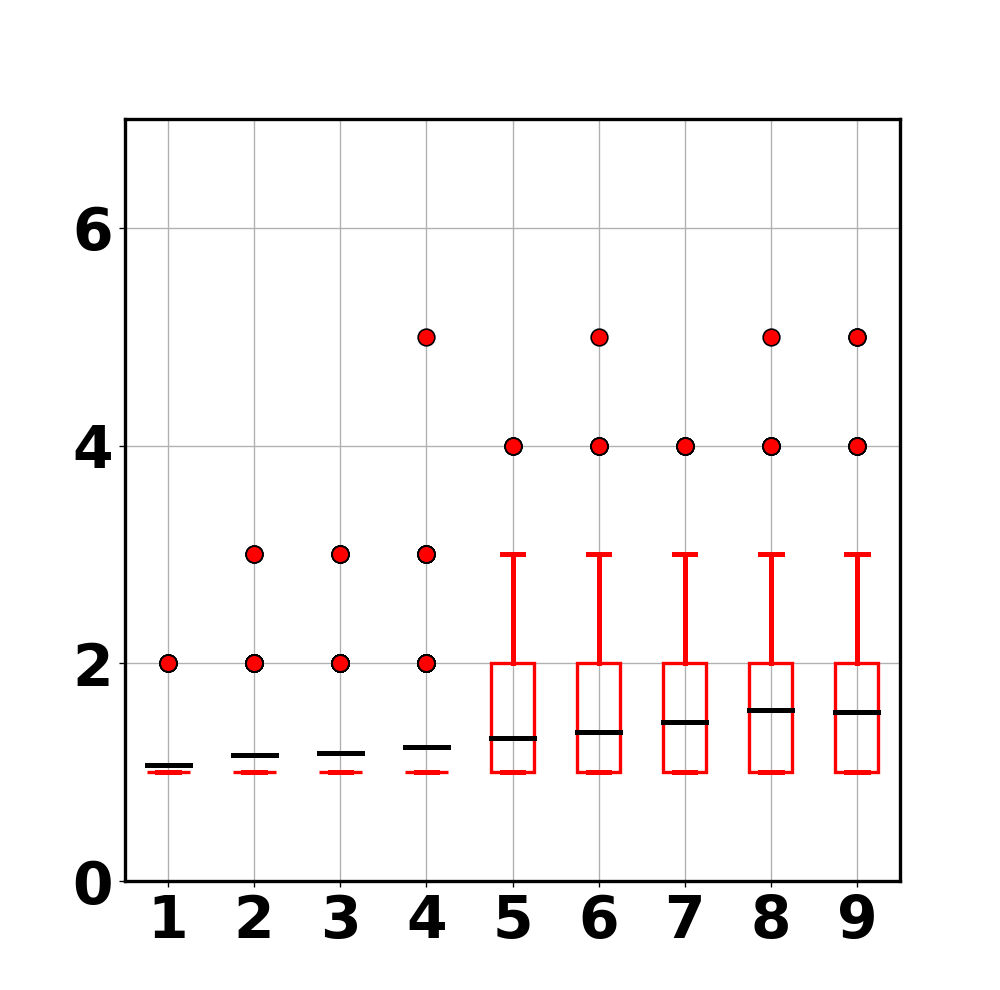}};
  \node[below of= img1, node distance=0cm, yshift=-2.0cm,font=\color{black}]  {\small  $100\sigma$ (veh./s/lane)};
  \node[left of= img1, node distance=0cm, rotate=90, anchor=center,yshift=1.8cm,font=\color{black}] {\small  $|\vcoord|$ (veh.)};
\end{tikzpicture}
}
\caption*{(a)}
\resizebox{4.8cm}{4.8cm}{
\begin{tikzpicture}
  \node (img1)  {\includegraphics[scale=0.16]{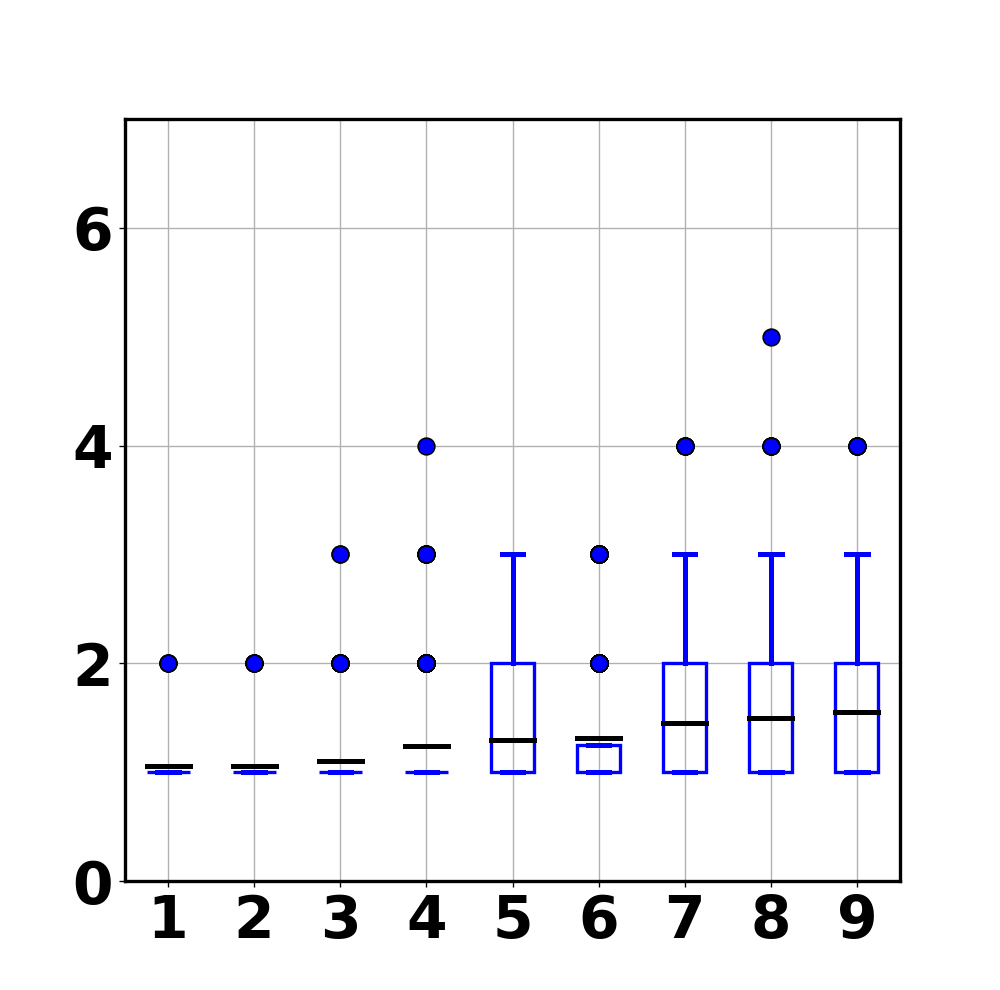}};
  \node[below of= img1, node distance=0cm, yshift=-2.0cm,font=\color{black}]  {\small  $100\sigma$ (veh./s/lane)};
  \node[left of= img1, node distance=0cm, rotate=90, anchor=center,yshift=1.8cm,font=\color{black}] {\small  $|\vcoord|$ (veh.)};
\end{tikzpicture}
}
\caption*{(b)}

\resizebox{4.8cm}{4.8cm}{
\begin{tikzpicture}
  \node (img1)  {\includegraphics[scale=0.164]{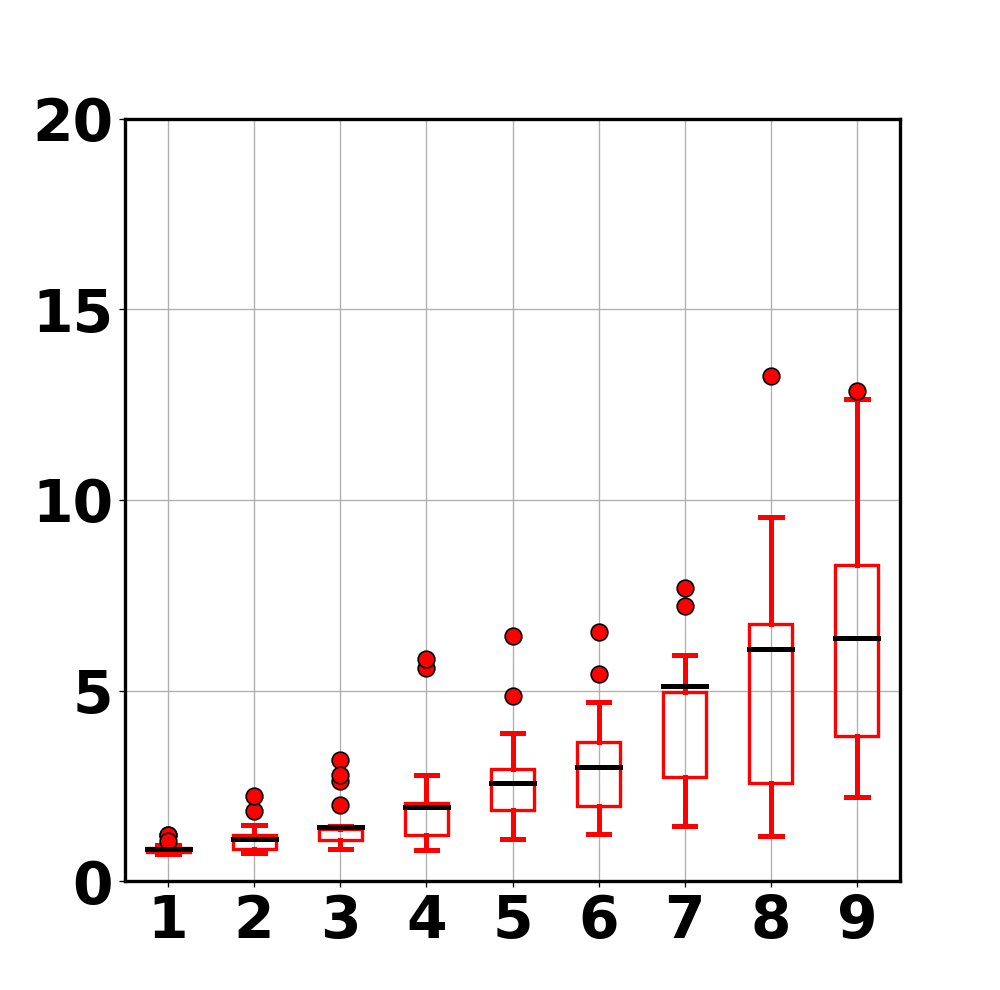}};
  \node[below of= img1, node distance=0cm, yshift=-2.0cm,font=\color{black}]  {\small  $100\sigma$ (veh./s/lane)};
  \node[left of= img1, node distance=0cm, rotate=90, anchor=center,yshift=1.9cm,font=\color{black}] {\small Comp. time/vehicle (s)};
\end{tikzpicture}
}
\caption*{(c)}
\resizebox{4.8cm}{4.8cm}{
\begin{tikzpicture}
  \node (img1)  {\includegraphics[scale=0.164]{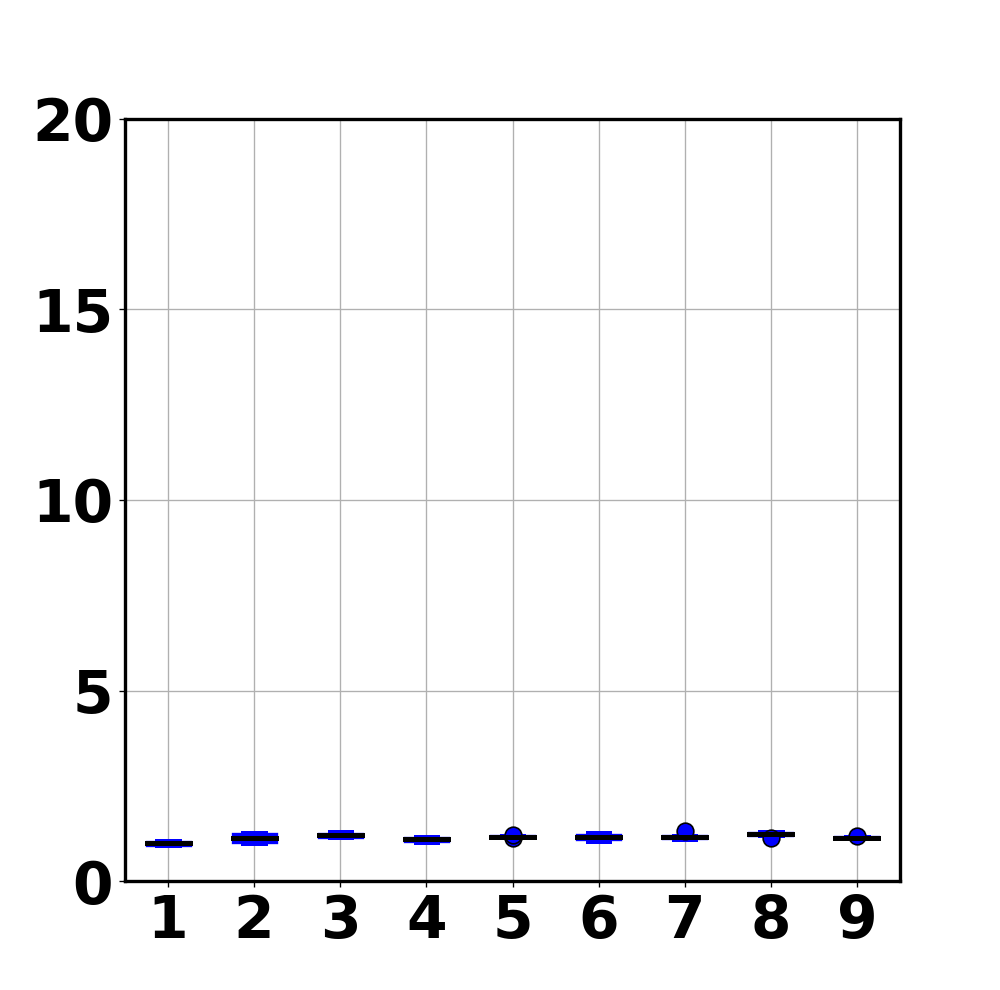}};
  \node[below of= img1, node distance=0cm, yshift=-2.0cm,font=\color{black}]  {\small $100\sigma$ (veh./s/lane)};
  \node[left of= img1, node distance=0cm, rotate=90, anchor=center,yshift=1.9cm,font=\color{black}] {\small  Comp. time/vehicle (s)};
\end{tikzpicture}
}
\caption*{(d)}
\end{multicols}
\caption{Results of the computation time comparison. Figures (a) and (c) correspond to combined optimization and (b) and (d) correspond to DD-SWA. 
In (a) and (b), the box plots 
represent the number of vehicles that participate in the trajectory optimization problem for the coordinated phase, which is denoted by $|\vcoord|$. In (c) and (d), computation time per vehicle is compared for combined optimization and DD-SWA. The lower and upper edge of the boxes represent the first and third quartile of the data respectively. The minimum and maximum of the data is represented by whiskers beyond the edges of the boxes. The outliers of the data are represented by circles beyond the whiskers. }
\label{fig:comptime}
\end{figure*}

\subsubsection*{Saturation of arrival rate}

In Comparison 2, at high arrival rates of traffic, the true arrival
rate of the vehicles is potentially reduced due to the feasibility
conditions mentioned in Theorem \ref{thm:feasiblity}. As mentioned
earlier, if the rear-end safety constraint is violated at the time of
arrival of a vehicle, the arrival time is delayed until the constraint
is satisfied. At high arrival rates, the rear-end safety constraint is
violated often. If necessary, the actual arrival of the vehicle is
delayed until the constraint is satisfied. In
Figure~\ref{fig:truesigma}, we present this idea by plotting the true
arrival rate per lane versus the set arrival rate per lane for DD-SWA.
We make use of the box plot to capture the variation of the true
arrival rate on the y-axis. This plot illustrates that the true
arrival rate saturates beyond $0.4$ vehicles/s per lane.  In Figure
\ref{fig:comp123}(b), the average TTC of vehicles in DD-SWA also
saturates at $0.4$ vehicles/s per lane which is consistent with the
saturation of the true arrival rate.


\begin{figure}[!htb]
\begin{tikzpicture}
  \node (img1) {\includegraphics[scale=0.18]{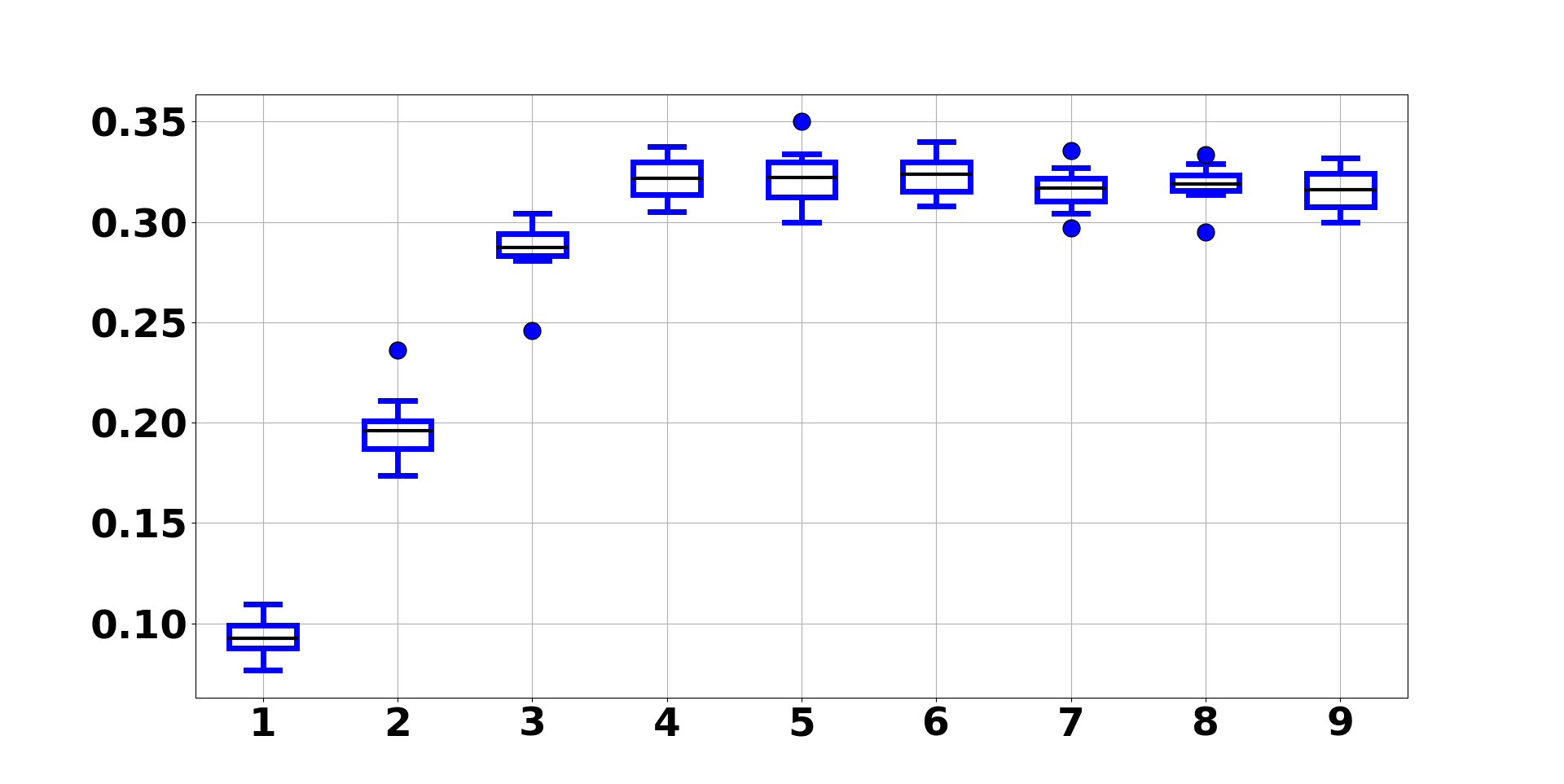}};
  \node[below of= img1, node distance=0cm, yshift=-2.3cm,font=\color{black}] { $10\sigma$ (veh./s/lane)};
  \node[left of= img1, node distance=0cm, rotate=90, anchor=center,yshift=4.25cm,font=\color{black}] { True $\sigma$ (veh./s/lane)};
  \label{fig:truearr} 
\end{tikzpicture}
\caption{The true arrival rate versus the desired $\sigma$ recorded
  for DD-SWA in Comparison 2. The upper and lower edges of the boxes
  represent the third and first quartile respectively. The whiskers of
  the boxes represent the maximum and minimum of the data.  The
  circles beyound the whiskers are the outliers.}
\label{fig:truesigma}
\end{figure}

\section{Conclusion}
In this work, we introduced a provably safe data-driven algorithm for
intersection management. By decomposing into two phases, we ensured
system wide safety and feasibility of vehicle
trajectories. Simulations suggest that the proposed algorithm performs
significantly better than traditional methods such as signalized
intersections and first-in first-out algorithms. We also demonstrated
through simulations that DD-SWA takes significantly less computational
effort compared to the centralized implementation with only marginal
loss in the objective value. Future work can be focused on developing
learning-based methodologies to automate the tuning of the parameters
in DD-SWA for various traffic scenarios. 
Other promising directions include extension to traffic management for
a network of intersections and hardware implementation on multi-robot
systems in regulated environments.

\bibliographystyle{IEEEtran}
\bibliography{../refer}


\end{document}